\documentclass[11pt]{article}
\usepackage{amsmath,amsthm,amsfonts,amssymb,mathrsfs,epsfig}
\usepackage[usenames]{color}
\usepackage{extarrows}
\usepackage[colorlinks=true]{hyperref}
\usepackage{graphicx}
\usepackage[scale=.8]{geometry}
\usepackage{ulem}
\usepackage{bbm}
\usepackage{textcomp}
\usepackage{amsmath}
\usepackage{esint}
\usepackage{cleveref}

\providecommand{\MR}{\relax\ifhmode\unskip\space\fi MR }

\providecommand{\href}[2]{#2}

\newtheorem{theorem}{Theorem}

\newtheorem{lemma}{Lemma}

\newtheorem{corollary}{Corollary}
\newtheorem{proposition}{Proposition}
\newtheorem{definition}{Definition}
\newtheorem*{assumption}{Assumption}

\theoremstyle{remark}
\newtheorem{remark}[theorem]{Remark}

\newtheorem{example}{Example}

 \def\beqlb{\begin{eqnarray}}\def\eeqlb{\end{eqnarray}}
 \def\beqnn{\begin{eqnarray*}}\def\eeqnn{\end{eqnarray*}}

 \def\qed{\hfill$\Box$\medskip}

\def\N{\mathbb{N}}

\def\R{\mathbb{R}}

\def\E{\mathbb{E}}

\renewcommand{\Phi}{\varPhi}
\renewcommand{\epsilon}{\varepsilon}

\renewcommand{\limsup}{\varlimsup}
\renewcommand{\liminf}{\varliminf}

\definecolor{mygray}{gray}{0.9}
\definecolor{deeppink}{RGB}{255,20,147}
\definecolor{mygreen}{rgb}{0.05, 0.576, 0.03}
\definecolor{myred}{rgb}{0.768, 0.09, 0.09}

\newcommand{\blue}{\color{blue}}

\long\def\symbolfootnote[#1]#2{\begingroup
\def\thefootnote{\fnsymbol{footnote}}\footnote[#1]{#2}\endgroup}

\usepackage{MnSymbol}

\usepackage{authblk}
\newcommand{\ddr}{\mathrm{d}}

\begin{document}

\title{\bf Weak convergence of continuous-state branching processes with large immigration}
\author{Cl\'ement Foucart\thanks{CMAP, Ecole Polytechnique, IP Paris, Palaiseau, France} \thanks{LAGA, Université Sorbonne Paris Nord, Villetaneuse, France \\
 \hspace*{1.5em} Email:  \href{mailto:clement.foucart@polytechnique.edu}{clement.foucart@polytechnique.edu}}\,\, and  Linglong Yuan \thanks{Department of Mathematical Sciences, University of Liverpool, Liverpool, UK\\
 \hspace*{1.5em} Email: \href{linglong.yuan@liverpool.ac.uk}{linglong.yuan@liverpool.ac.uk}}}

\maketitle

\begin{abstract} 
Functional limit theorems are established for continuous-state branching processes with immigration (CBIs), where the reproduction laws have finite first moments and the immigration laws exhibit large tails. Different regimes of immigration are identified, leading to limiting processes that are either subordinators, CBIs, extremal processes, or extremal shot noise processes.
\end{abstract}
 \vspace{9pt} \noindent {\bf Key words.}
{Continuous-state branching process}, {Immigration}, {Functional limit theorem}, {Subordinator}, {Extremal process}, {Shot noise process}, {Regularly varying function}.


\section{Introduction and main results}\label{CBI}
The objective of the article is to establish some functional limit theorems for continuous-state branching processes with immigration (CBIs).  Those positive Markov processes have been defined in the seventies by Kawazu and Watanabe \cite{KAW} and can be seen as analogues of Galton-Watson processes with immigration (GWIs) in continuous time and space. 

CBI processes are characterised in law by two Lévy-Khintchine functions $\Psi$ and $\Phi$. The branching mechanism $\Psi$ is the Laplace exponent of a spectrally positive Lévy process and governs the dynamics of reproduction in the population. The mechanism of immigration $\Phi$ is the Laplace exponent of a subordinator (i.e.\ an increasing Lévy process). More background on CBIs and recalls on their semigroup, infinitesimal generator and their Poissonian construction, are provided in Section \ref{sec:backgroundCBI}.

The question of how to renormalise, possibly in a nonlinear way, the one-dimensional marginals of CBI processes has been addressed in  Foucart et al. \cite{zbMATH07549543}. We refer also to the references therein for previous works on this topic, see in particular Pakes \cite{Pakes} and Pinsky \cite{Pinsky}.  Recently, a functional limit theorem for GWIs has been established by Iksanov and Kabluchko, see \cite{IKSANOV2018291}. They found an interesting regime of immigration for which the limiting process arising is an \textit{extremal shot noise} process. The latter is a generalisation of the extremal processes (i.e.\ the successive records of a Poisson point process) which incorporates a drift mechanism. Among other things, we will establish the analogue statement as \cite[Theorem 1.1]{IKSANOV2018291} for CBIs. Our proof is based  on different arguments and enables us to address other regimes as well.

Extremal shot noise processes (ESNs for short) have been studied by Dombry \cite{Dombry} in the setting of random fields. They have been then explored from the perspective of Markov processes in \cite{FoucartYuanESN23}. Their definition is recalled later in this section and more background on them is provided in Section \ref{sec:recallESN}.


In this article, we assume that the branching dynamics has a finite mean, i.e.\ $\Psi'(0+)>-\infty$ and we consider mainly immigration with infinite mean, that is to say $\Phi'(0+)=\infty$, or equivalently $\int_{1}^{\infty}h\nu(\ddr h)=\infty$, with $\nu$ the L\'evy measure associated with $\Phi$. Throughout the article, we refer to $\nu$ as the \textit{immigration measure}. Recall that it satisfies L\'evy's integrability condition,  $\int_0^\infty 1\wedge h\, \nu(\ddr h)<\infty$.

We further distinguish two regimes of immigration. The immigration measure $\nu$ is said to have a log moment when $\int_1^{\infty} \ln h \, \nu(\ddr h)<\infty$ and no log moment if this integral is infinite. As a matter of fact, the following equivalence holds and will play a role hereafter, see e.g.\ \cite[Remark 7]{zbMATH07549543}: 
\begin{equation}\label{logmomentwithphi}
    \int_1^{\infty} \ln h \, \nu(\ddr h)=\infty \text{ if and only if } \int_0^{\theta} \frac{\Phi(x)}{x}\ddr x=\infty \text{ for some } \theta>0.
\end{equation}
 
As observed in \cite{Pinsky} and \cite[Theorems 1 and 2]{zbMATH07549543}, see also Keller-Ressel and Mijatovic \cite[Theorem 2.6]{KELLERRESSEL20122329}, when the mechanisms of the CBI satisfy the integrability condition $\int_{0}^{\theta}\frac{\Phi(x)}{|\Psi(x)|}\ddr x<\infty$ for some\footnotemark
\footnotetext{and then for all small $\theta>0$} 
 $\theta>0$,  the immigration and the branching dynamics align in some sense. Indeed, in this case, when $\Psi$ is (sub)-critical, i.e. $\Psi'(0+)\geq 0$, the process is positive recurrent. Conversely, when $\Psi$ is supercritical, i.e. $\Psi'(0+)<0$,  the CBI process grows towards $\infty$ along the same rate as the pure branching process with mechanism $\Psi$ on its event of non-extinction. 
\medskip

On the contrary, when for some $\theta>0$, $\int_0^\theta\frac{\Phi(x)}{|\Psi(x)|}\ddr x=\infty$, the immigration becomes highly active and can dominate the branching dynamics. 
As first noticed in \cite{Pinsky}, and further explored in \cite{zbMATH07549543}, the extent of this ``overtaking" actually depends on how fast the integral $\int_\epsilon^\theta \frac{\Phi(x)}{|\Psi(x)|}\ddr x$ tends to $\infty$ as $\epsilon \rightarrow 0$. We will focus on this setting.
Notice that when $\Psi$ is non-critical i.e.\ $\Psi'(0+)\neq 0$, plainly by \eqref{logmomentwithphi}, one has the equivalence: $\int_0^{\theta} \frac{\Phi(x)}{|\Psi(x)|}\ddr x=\infty$ for some $\theta>0$, if and only if $\nu$ has no log moment. We stress moreover that only in the critical case, i.e.\ $\Psi'(0+)=0$, the immigration measure can have a log moment and a mechanism $\Phi$ that satisfies $\int_0^{\theta} \frac{\Phi(x)}{|\Psi(x)|}\ddr x=\infty$. 
\medskip

The limiting processes that arise in the functional limits fall into four categories. When $\nu$ has a log moment, the processes obtained at the limit are  subordinators or CBIs. Typically, when $\nu$ does not have a log moment, extremal processes or extremal shot noise processes emerge.  

\smallskip

\textbf{Notation.} We denote by $\mathbb{D}$ the Skorokhod space, that is the space of c\`adl\`ag functions from $\mathbb{R}_{+}:=[0,\infty)$ to $\mathbb{R}_{+}$ equipped with the Skorokhod distance. We refer to Billingsley's book \cite{zbMATH01354815} for those notions. The weak convergence in the Skorokhod sense is denoted by $\Longrightarrow$. The convergence in the sense of the finite-dimensional distributions is denoted by $\overset{\text{f.d.d.}}{\Longrightarrow}$. For any Borel measure $m$ on $(0,\infty)$, we denote its tail by $\bar{m}(x):=m\left((x,\infty)\right)$ for all $x>0$. We take the convention $1/0=\infty$ and $1/\infty =0$. For all real numbers $a,b$, we set $a\vee b:=\max\{a,b\}$ and $a\wedge b:=\min\{a,b\}$. For any injective function $f$, we denote by $f^{-1}$ its inverse function. For any function $f$, let $\|f\|$ denote the uniform norm of $f$. 
The space $C^\infty$ denotes the functions that have derivatives of all orders, all continuous.  Last, for any positive function $f$, we write $\int_0 f(x)\ddr x<\infty$ when $f$ is integrable on an interval $(0,\theta)$ for any small enough $\theta>0$. Similarly  $\int^\infty f(x)\ddr x<\infty$ means that $f$ is integrable on $(\theta,\infty)$ for any large enough $\theta>0$.
\medskip

Recall that a function $\varphi$ is said to be regularly varying at $\infty$ (resp. at $0$) if $\varphi(\lambda x)/\varphi(x)\rightarrow \lambda^{\alpha}$ as $x$ goes to $\infty$ (resp. at $0$) for any $\lambda>0$ and some index $\alpha\in \mathbb{R}$. When $\alpha=0$, the function $\varphi$ is slowly varying. If $\varphi$ is strictly increasing, regularly varying with index $\alpha>0$, $\varphi^{-1}$ is regularly varying with index $1/\alpha$. Regularly varying functions and their fundamental properties, such as Karamata's representation theorem and the uniform convergence theorem, will play an important role in this work. We refer the reader to Bingham et al. \cite{regularvariation} for those notions.


\medskip 

We state first a proposition focusing on regularly varying mechanisms. In this context, the branching mechanism $\Psi$ is critical, the CBI processes can be linearly renormalised and their limiting processes are either subordinators or CBIs.
\newpage
\begin{proposition}\label{thm2} Let $(Y_t,t\geq 0)$ be a $\mathrm{CBI}(\Psi,\Phi)$ whose mechanisms satisfy
\begin{center}
$\Psi(x)\underset{x\rightarrow 0}{\sim} dx^{1+\alpha}\ell(x)$ and $\Phi(x)\underset{x\rightarrow 0}{\sim}d'x^{\beta}\ell(x)$, 
\end{center} 
for some function $\ell$, slowly varying at $0$, $d\geq 0,d'>0$ and $\alpha,\beta\in (0,1]$ such that $\beta\leq \alpha$. 
\smallskip

Then
\begin{equation}
\label{linearFLT} \big(\Phi^{-1}(1/t)Y_{st},s\geq 0\big)\underset{t\rightarrow \infty}{\Longrightarrow} (X_s,s\geq 0),\end{equation}
where $(X_s,s\geq 0)$ is a CBI$(\bar{\Psi},\bar{\Phi})$ started from $0$ with 
\[\begin{cases}
& \bar{\Psi}(x):=\frac{d}{d'}x^{1+\alpha} \text{ and } \bar{\Phi}(x):=x^{\alpha} \text{ if } \beta=\alpha,\\
& \bar{\Psi}(x):\equiv 0 \text{ and } \bar{\Phi}(x):=x^{\beta} \text{ if } \beta<\alpha.
\end{cases}\]
\end{proposition}
\begin{remark}
\begin{enumerate}
    \item[i)]  In all cases of Proposition \ref{thm2}, by the assumption $\beta\leq \alpha$, we have $\int_0 \frac{\Phi(x)}{\Psi(x)}\ddr x=\infty$.  Furthermore, when $\beta<1$, the  immigration measure possesses a log moment. Indeed, we have plainly $\int_{0}\frac{\Phi(x)}{x}\ddr x<\infty$, which by \eqref{equivphinuintro} allows us to conclude. In the case $\beta=\alpha=1$, the finiteness of the log moment depends on whether $\int_0 \ell(x)\ddr x<\infty$ or not.
    \item[ii)]   When $\beta<\alpha$ or $d=0$, the limiting process $X$ is a CBI$(0,\bar{\Phi})$, namely a stable subordinator with Laplace exponent $\bar{\Phi}$.
    \item[iii)] In the case $\beta>\alpha$, we have  $\int_0 \frac{\Phi(u)}{\Psi(u)}\ddr u<\infty$ and the CBI admits a limiting distribution.
Regarding functional limits of the form \eqref{linearFLT}, this is a degenerate case,  as for any fixed $s>0$, $Y_{st}$ converges in law  towards the limiting distribution as $t$ goes to $\infty$.
\end{enumerate}
\end{remark}
\begin{remark}\label{remlocalextinction} 
The CBI$(\bar{\Psi},\bar{\Phi})$, $X$, is self-similar of index $\alpha$, see e.g. Patie \cite[Lemma 4.8]{Patie}.  According to Duhalde et al. \cite[Example 1]{Duhalde}, $X$ is recurrent (resp. transient) if $d'/d\leq \alpha$ (resp. $d'/d>\alpha$). Moreover, $X$ hits $0$ with positive probability if and only if $d'/d<\alpha$. In this case, heuristically, there are random times $s$, at which the process $(Y_{st},t\geq 0)$ stays smaller than $1/\Phi^{-1}(1/t)$. 
\end{remark}
\medskip

When the immigration measure $\nu$ has no log moment, it has been observed in \cite[Theorem 3 and Corollary 1]{zbMATH07549543} that when the branching mechanism is non-critical, no linear renormalisation of the CBI process is possible. 
This can be explained by the fact that some immigration arrivals are so large that they may overwhelm the extant population and change radically the rate of its growth. 
%
This phenomenon is reflected by some extremal behaviours in the CBI processes.
\smallskip

We shall indeed find conditions on the immigration mechanism $\Phi$, see the forthcoming assumption $\mathbb{H}$, under which the limiting process that will arise is a so-called ESN process. The latter takes  the following form
\begin{equation}\label{maxidform}
\left(M_s,s\geq 0\right):=\big(\sup_{0\leq u\leq s}\big(\xi_u-\gamma (s-u)\big), s\geq 0\big),
\end{equation}
where $\gamma \in \mathbb{R}$ and $\mathcal{N}:=\sum_{u\geq 0}\delta_{(u,\, \xi_u)}$ is a Poisson point process (in short PPP) with intensity $\ddr t\times \ddr \mu$, for some measure $\mu$ on $(0,\infty)$ such that $\bar{\mu}(x)<\infty$ for any $x>0$. 
\smallskip

The process $M$ in \eqref{maxidform} is an ESN process based on the measure $\mu$ with slope $-\gamma $. We write for short $\mathrm{ESN}(\gamma ,\mu)$. When $\gamma=0$, $M$ in \eqref{maxidform} is a classical extremal process based on the measure $\mu$. We refer 
to Resnick's book \cite[Chapter 4.3]{Res87}. 
\medskip

We now introduce and explain our main assumption regarding the immigration mechanism $\Phi$. Under this assumption, after renormalisation of the CBI, extremal shot noise processes will emerge as the limiting processes.
\begin{assumption}[$\mathbb{H}$] 
\begin{align*}
x\mapsto x\Phi(e^{-x}) &\text{ is regularly varying at } \infty \text{ and has a non-decreasing equivalent function}.
\end{align*}
\end{assumption}
\begin{remark}
   Assumption $\mathbb{H}$ implies that $\Phi$ is slowly varying at $0$ and by Karamata Tauberian theorem and monotone density theorem, see the proof of forthcoming Lemma \ref{lem: fastjumpdiff} for details and references, one has
    \begin{equation}\label{equivphinuintro}
        \Phi(1/y)\underset{y\rightarrow \infty}{\sim} \bar{\nu}(y).
    \end{equation}
    Observe also that, when $\mathbb{H}$ holds, $\int^{\infty}\Phi(e^{-x})\ddr x=\infty$, since
    $c=\underset{x\rightarrow \infty}{\lim} x\Phi(e^{-x})\in (0,\infty]$. By the change of variable $u=e^{-x}$, we get $\int_0 \frac{\Phi(u)}{u}\ddr u=\infty$, so that $\nu$ has no log moment. 
\end{remark}
\begin{remark} 
The regular variation assumption in  $\mathbb{H}$ can be written more explicitly as follows: 
\begin{align}\label{regularvariationassumption}  \qquad \qquad & \exists \delta \in \mathbb{R} \text{ and }  \ell \text{ slowly varying at } 0 \text{ such that }
x\mapsto x\Phi(e^{-x})\underset{x\rightarrow \infty}{\sim} x^{1-\delta}\ell(x).
\end{align}
The facts that $x\mapsto \Phi(e^{-x})$ is decreasing and that $x\Phi(e^{-x})\sim H(x)$ as $x$ goes to $\infty$, where $H$ is a non-decreasing function, imply that $\delta\in [0,1]$.  Notice that if $\delta \in [0,1)$, i.e. $x\mapsto x\Phi(e^{-x})$ is regularly varying with positive index, then the existence of a non-decreasing equivalent function $H$ is guaranteed, see \cite[Theorem 1.5.3, page 22]{regularvariation}. When $\delta=1$, the map $x\mapsto x\Phi(e^{-x})$ is slowly varying at $\infty$ and we impose further in Assumption $\mathbb{H}$  that there exists a non-decreasing equivalent function $\ell$. 
\end{remark}
\medskip

The assumption $\mathbb{H}$ covers two regimes:
\begin{align}\label{tworegimes}
         \textbf{Log case}: \ x\Phi(e^{-x})\underset{x\rightarrow \infty}{\longrightarrow} c \in (0,\infty), \quad \text{ or } & \textbf{ Super-log case}: \ x\Phi(e^{-x})\underset{x\rightarrow \infty}{\longrightarrow} \infty.  
         \end{align}
\noindent When $x\mapsto x\Phi(e^{-x})$ is regularly varying but not slowly varying, i.e. $\delta\in [0,1)$ in \eqref{regularvariationassumption}, then we are in the Super-log regime.  The slowly varying case, for which $\delta=1$ in \eqref{regularvariationassumption}, lies at the boundary between the two regimes. It depends on whether the limit $\underset{x\rightarrow \infty}{\lim} x\Phi(e^{-x})=\underset{x\rightarrow \infty}{\lim} \ell(x)$ is finite or not (the latter exists in $(0,\infty]$ by assumption $\mathbb{H}$).  
\medskip

The main contribution of this article is the following theorem.
\begin{theorem} \label{thmdivergencecase}
Let $(Y_t,t\geq 0)$ be a $\mathrm{CBI}(\Psi,\Phi)$ with $\Psi\not \equiv 0$ and $b:=\Psi'(0+)>-\infty$. Assume $\mathbb{H}$. 

Then,
\begin{equation}\label{eqn:logconv}\left(\frac{1}{t\Phi(1/Y_{st})},s\geq 0\right)\underset{t\rightarrow \infty}{\Longrightarrow}  (M_s,s\geq 0),\end{equation}
where $(M_s,s\geq 0 )$ is an $\mathrm{ESN}(b/c,\mu)$ with $c:=\underset{x\rightarrow \infty}{\lim} x\Phi(e^{-x})\in (0,\infty]$ and $\mu(\ddr x):=\ddr x/x^2$.
\smallskip

The convergence \eqref{eqn:logconv} also holds if we replace $\Phi$ by an increasing function $F$ equivalent to $\Phi$ at $0$.
\end{theorem}
\begin{remark}\label{rem:b=0c=infinity} 
\begin{enumerate}
\item Theorem \ref{thmdivergencecase} includes also the critical branching mechanism for which $b=0$. This completes \cite[Section 3.4]{zbMATH07549543} where the one-dimensional laws (i.e. with $s=1$) were studied under extra assumptions on $\Psi$. 
\item Observe that when $b=0$ or $c=\infty$, the ESN process $M$ shrinks to a classical extremal process based on the measure $\mu$. 
\end{enumerate}
 \end{remark}
\begin{remark}\label{remarknoconv}
    We shall see along the proof of Theorem \ref{thmdivergencecase}, that when $x\Phi(e^{-x})\underset{x\rightarrow \infty}{\longrightarrow}0$, (call it Sub-log case), the convergence of the form \eqref{eqn:logconv} cannot hold.
\end{remark}

\medskip
We describe now into more details the renormalisation function and the convergence \eqref{eqn:logconv}.
\smallskip

\indent In the Log case, see \eqref{tworegimes}, we have \begin{equation}\label{logcaseequiv}\Phi(1/y)\underset{y\rightarrow \infty}{\sim} \frac{c}{\ln(1+y)}.\end{equation}
Equivalently, by \eqref{equivphinuintro},  $\bar{\nu}(u)\underset{u\rightarrow \infty}{\sim} \frac{c}{\ln (1+u)}$. Set $(\tilde{M}_s,s\geq 0):=(cM_s,s\geq 0)$. The convergence \eqref{eqn:logconv} can  be  rewritten more explicitly as follows
    \[\left(\frac{1}{t}\ln\left(1+Y_{st}\right),s\geq 0\right) \underset{t\rightarrow \infty}{\Longrightarrow}  (\tilde{M}_s,s\geq 0).\]
 The process $\tilde{M}$ corresponds to the ESN process with slope $-b$ constructed as in \eqref{maxidform} from the PPP $\tilde{\mathcal{N}}$ which is obtained as the image of $\mathcal{N}$ by $\xi\mapsto c\xi$. The latter has for intensity measure $\tilde{\mu}(\ddr x):=c/x^{2}\ddr x$. We recover here the analogue result as obtained for the GWIs in \cite[Theorem 1.1]{IKSANOV2018291}. 
\medskip

\begin{remark} The ESN process $M$ (or equivalently $\tilde{M}$ above) is self-similar with index $1$, see \cite[Example]{FoucartYuanESN23}. It is transient if and only if $b\leq 0$, or $b>0$ and $c/b>1$. Moreover, $M$ hits $0$ if and only if $b>0$ and $c/b<1$. Similarly as in Remark \ref{remlocalextinction}, intuitively, this indicates that at certain random times $s$, $(Y_{st},t\geq 0)$ does not evolve at an exponential scale.\end{remark}

\medskip
\indent In the Super-log case, see \eqref{tworegimes}, we have 
 \begin{equation}\label{superlogequivalent} 
 \Phi(1/y)\underset{y\rightarrow \infty}{\sim} \frac{\ell\big(\ln(1+ y)\big)}{\ln(1+y)^{\delta}}
 \end{equation} 
 with $\delta\in [0,1]$ and $\ell$ a slowly varying function at $\infty$ such that:
 if $\delta=0$ then $\underset{y\rightarrow \infty}{\lim} \ell(y)=0$, since $\lim_{x\to\infty}\Phi(e^{-x})=0$; if  $\delta =1$, $\ell$ is non-decreasing (by the assumption $\mathbb{H}$) and  $\underset{y\rightarrow \infty}{\lim} \ell(y)=\infty$.
 The equivalence \eqref{superlogequivalent} can be written in terms of the tail of $\nu$, see \eqref{equivphinuintro}, as follows:
 \begin{equation}
     \label{equivnu} \bar{\nu}(u)\underset{u\rightarrow \infty}{\sim} \ln(1+u)^{-\delta}\ell (\ln(1+ u)).
 \end{equation}
 

\begin{example}[see Example 1-(3) in \cite{zbMATH07549543}]
Let $M$ be an extremal process based on the measure $\mu(\ddr x)=\ddr x/x^2$. Both cases below are in the Super-log regime: 
\begin{enumerate}
    \item If $\bar{\nu}(u)\underset{u\rightarrow \infty}{\sim} 1/\ln (1+\ln(1+u))$ then 

\[\left(\frac{1}{t}\ln (1+\ln (1+Y_{st})),s\geq 0\right)\underset{t\rightarrow \infty}{\Longrightarrow}  (M_s,s\geq 0).\]
    \item If $\bar{\nu}(u)\underset{u\rightarrow \infty}{\sim} \ln (1+\ln (1+u)) /\ln(1+ u)^{\delta}$, with $\delta\in (0,1]$ 
    then 
\[\left(\frac{1}{t}\frac{\ln (1+Y_{st})^{\delta}}{\ln (1+\ln (1+Y_{st}))},s\geq 0\right)\underset{t\rightarrow \infty}{\Longrightarrow}  (M_s,s\geq 0).\]
Notice that $\mathbb{H}$ is fulfilled in the case $\delta=1$, since one can take $\ell(y):=\ln(1+y)$ which is increasing.
\end{enumerate}
\end{example}


We state now an aside result for subordinators. The latter can be seen as pure immigration CBI processes, i.e. with branching mechanism $\Psi\equiv 0$.  
\begin{theorem}[Subordinators]\ \label{thmextremasubordinator} Let $(Y_t,t \geq 0)$ be a subordinator with Laplace exponent $\Phi$.
\begin{itemize}
\item[(i)] If $\Phi$ is slowly varying at $0$, then 
\begin{equation}\label{convsubordinator}
\left(\frac{1}{t\Phi(1/Y_{st})},s\geq 0\right)\xLongrightarrow[t\rightarrow \infty]{\text{f.d.d.}} (M_s,s\geq 0),
\end{equation}
where $(M_s,s\geq 0)$ is an extremal process based on the measure $\mu(\ddr x)=\ddr x/x^2$.

\item[(ii)] Under the assumption $\mathbb{H}$, the convergence \eqref{convsubordinator} holds in the Skorokhod sense and  one can replace $\Phi$ by any increasing function $F$ equivalent to $\Phi$ at $0$.
\end{itemize}
\end{theorem}
\begin{remark} We see that the renormalisation function and the limiting process are the same for the subordinators and the CBIs in the Super-log case or in the Log case when $b=0$, see Remark \ref{rem:b=0c=infinity}. This reflects the fact that immigration drives the growth of the population in those particular regimes.
\end{remark}
\begin{remark}
Theorem \ref{thmextremasubordinator} is reminiscent to a functional limit theorem near time $0$ for subordinators obtained by  Kella and L\"opker \cite{KELLA20133122} in the case $\Phi(x)\sim \gamma \ln x$ as $x$ goes to $\infty$ for some $\gamma>0$. We refer also to Maller and Schindler \cite{MALLER20194144} for another related study. 
\end{remark}
We have left unaddressed the case of a branching dynamics with infinite mean $(b=-\infty)$ as well as the scenario where $c=\underset{x\rightarrow \infty}{\lim}x\Phi(e^{-x})=0$, which is excluded by our assumption $\mathbb{H}$.  In these situations, the $\mathrm{ESN}(b/c,\mu)$ process is no longer defined. It is noteworthy that when $b=-\infty$, the growth of the CB process is almost surely super-exponential, leading to certain extremal processes, as discussed in Foucart and Ma \cite[Theorem 1]{foucart2019} and the references therein. The problem of designing  functional limit theorems of CBIs in this case is left for potential future works.
\section{Background on CBI and ESN processes}\label{secCBI}
The proofs of Theorem \ref{thm2} and Theorem \ref{thmdivergencecase} will be based on the convergence of generators. We gather here fundamental properties of generators of CBIs and ESNs. We refer the reader to \cite{KAW} or to Li's lecture notes \cite{zbMATH07273637} for a more recent account on CBIs, see \cite{FoucartYuanESN23} for background on ESNs.
\subsection{Continuous-state branching processes with immigration}\label{sec:backgroundCBI}
We shall focus on CBI processes whose branching mechanism has a finite mean. In this case, the function $\Psi$ takes the form
\begin{equation}\label{eqpsi}
\Psi(q)=bq+\frac{\sigma^2}{2}q^2+\int_{0}^{\infty}(e^{-qx}-1+qx)\pi(\ddr x),
\end{equation}
with $b\in \mathbb{R}$, $\sigma\geq 0$ and $\pi$ a Lévy measure on $(0,\infty)$ satisfying $\int_0^\infty x\wedge x^2 \, \pi(\ddr x)<\infty$. 
The mechanism of immigration takes the form
\begin{equation}\label{eqpsi}
\Phi(q)=\beta q+\int_{0}^{\infty}(1-e^{-qx})\nu(\ddr x),
\end{equation}
with $\beta\in \mathbb{R}_+$ and $\nu$ a Lévy measure on $(0,\infty)$ satisfying $\int_0^\infty 1\wedge x \, \nu(\ddr x)<\infty$.

The semigroup of a $\mathrm{CBI}(\Psi,\Phi)$ is characterised by its Laplace transform as follows, see e.g. \cite[Theorem 5.6]{zbMATH07273637}: for $\lambda\in \mathbb{R}_{+}$ and $x\in \mathbb{R}_{+}$,  \begin{equation}\label{cumulantCBI}
\mathbb{E}_x[e^{-\lambda Y_{t}}]=\exp \big(-xv_{t}(\lambda)-\int_{0}^{t}\Phi(v_{s}(\lambda))\ddr s\big),
\end{equation}
where the function $t\mapsto v_{t}(\lambda)$ is the solution to 
\begin{equation}\label{cumulant} \frac{\partial}{\partial t}v_{t}(\lambda)=-\Psi(v_{t}(\lambda)), \quad v_{0}(\lambda)=\lambda. \end{equation}
In particular, the CBI$(\Psi,\Phi)$ is a Feller process. When $\Psi\equiv 0$, namely when there is no branching, $v_t(\lambda)=\lambda$ for all $t$, and the Markov process $Y$ with semigroup \eqref{cumulantCBI} is nothing but a subordinator with Laplace exponent $\Phi$.  For some background of subordinators, we refer the reader for instance to  Kallenberg's book \cite[Chapter 15]{Kallenberg}. When $\Phi\equiv 0$, the process $(Y_t,t\geq 0)$ is a continuous state branching process with mechanism $\Psi$, in short CB$(\Psi)$. 

%
Denote by $\mathcal L$ the generator of the CBI$(\Psi,\Phi)$ and by $\mathcal{D}_\mathcal{L}$ its domain. Call $C^2[0,\infty)$ the set of bounded continuous real functions on $[0,\infty)$ with bounded continuous derivatives up to the second order. Denote also by $\mathcal{D}_{\mathrm{CBI}}:=\mathrm{Vect}\{e_z, z\geq 0\}$ the linear span of exponential functions $e_z:x\mapsto e_z(x):=e^{-xz}$ defined on $[0,\infty)$.  For any $f\in C^2[0,\infty)$ and $x\in [0,\infty)$, see e.g. \cite[Equation (7.3)]{zbMATH07273637}, 
\begin{align}
\mathcal L f(x)&:=x\mathrm{L}^{\Psi}f(x)+\mathrm{L}^{\Phi}f(x) \nonumber \\
&:=\frac{\sigma^2}{2}xf''(x)-bxf'(x)+x\int_0^\infty\Big(f(x+u)-f(x)-uf'(x)\Big)\pi(\ddr u) \label{branchingpartgen}\\
&\qquad +\beta f'(x)+\int_0^\infty\Big(f(x+u)-f(x)\Big)\nu(\ddr u), \label{immigrationpartgen}
\end{align}
where $\mathrm{L}^{\Psi}$ and $\mathrm{L}^{\Phi}$ denote respectively the generator of a spectrally positive Lévy process with Lévy-Khintchine function $\Psi$ and that of a subordinator with Laplace exponent $\Phi$. Moreover, one has for any $z\in \mathbb{R}_+$ and all $x\in \mathbb{R}_+$:
\begin{equation}\label{genCBIonexp}
\mathcal L e_z(x)=\big(x\Psi(z)-\Phi(z)\big)e_z(x).
\end{equation}
We shall need the following lemma providing a core and a subspace of the domain of CBIs.
\begin{lemma} Let $C^{2,0}([0,\infty)):=\{f\in C^2([0,\infty)): \underset{x\rightarrow \infty}\lim f(x)=0 \text{ and } \underset{x\rightarrow \infty}\lim x(|f''(x)|+|f'(x)|)=0\}$, then
\[\mathcal{D}_{\mathrm{CBI}}\subset C^{2,0}([0,\infty))\subset \mathcal{D}_\mathcal{L}\]
and $\mathcal{D}_{\mathrm{CBI}}$ is a core for the CBI$(\Psi,\Phi)$.
\end{lemma}
\begin{remark}
Note that $C^{2}_c([0,\infty))$, the space of twice differentiable functions with compact support in $[0,\infty)$, is included in $C^{2,0}([0,\infty))$ and then also in $\mathcal{D}_{\mathcal{L}}$.
\end{remark}
\begin{proof}
The first inclusion is trivial. For the second, according to \cite[Theorem 7.2-(4)]{zbMATH07273637}, the process 
\[(W_t,t\geq 0):=\left(f(Y_t)-\int_{0}^{t}\mathcal{L}f(Y_s)\ddr s,t\geq 0\right)\] is a local martingale for any $f\in C^2([0,\infty))$. If moreover, $f\in C^{2,0}([0,\infty))$, then $\sup_{x\in [0,\infty)}|\mathcal{L}f(x)|<\infty$ and the integrand $s\mapsto\mathcal{L}f(Y_s)$ is bounded over any compact interval of time almost surely. This entails that $W$ is a true martingale, and since $f$ and $\mathcal{L}f$ vanish at $\infty$, by the reverse Dynkin's formula, $f$ lies in the domain of $Y$, see e.g. \cite[Proposition 1.7, Chapter VII]{RevuzYor}. Last, clearly by \eqref{cumulantCBI}, the set $\mathcal{D}_{\mathrm{CBI}}$ is invariant for the semigroup. It is moreover a dense subset of $C_0([0,\infty))$, the space of continuous functions vanishing at $\infty$, and by applying \cite[Proposition 19.9]{Kallenberg}, we see thus that it is a core. 
\end{proof}
%
\subsection{Poisson shot noise structure and heuristics}
We describe here  briefly the Poissonian construction of the CBI processes and give heuristic arguments explaining the convergence in Theorem \ref{thmdivergencecase}.

Any CBI process admits a Poisson \textit{shot noise} structure, 
see e.g. Li \cite[Theorem 6.4 and Theorem 6.7]{zbMATH07273637}. When there is no drift term in $\Phi$, i.e. $\beta=0$ (no continuous immigration), the latter takes a simple form. Indeed the CBI$(\Psi,\Phi)$ started from $0$ can be constructed in a Poissonian way as follows
\begin{equation}\label{idform} \left(Y_s,s\geq 0\right)=\left(\sum_{0\leq u\leq s}X^{u}_{s-u},s\geq 0\right),
\end{equation}
where $\mathcal{M}:=\sum_{u\geq 0}\delta_{(u,X^u)}$ is a PPP on $[0,\infty)\times \mathbb{D}$, with intensity $\ddr s\times n(\ddr X)$ and $n$ is the measure on $\mathbb{D}$ defined by \[n(\ddr X)=\int_{0}^{\infty}\nu(\ddr x)\mathbb{P}^{\Psi}_{x}(\ddr X),\] with $\mathbb{P}^{\Psi}_{x}$ denoting the law of a c\`adl\`ag CB$(\Psi)$ started from $x>0$.  
Immigration arrivals occur along the atoms of times of $\mathcal{M}$ and the Borel measure $\nu$ governs the  initial amount of individuals $X_0^u$ at an immigration time $u$.   

We now explain how the Poissonian construction \eqref{idform} provides an insight into the emergence of extremal behaviours in the limit. 

By assumption, all the grafted CB$(\Psi)$ processes, $X^u$, have a finite mean. When starting from macroscopic initial values, their order, on the interval of time $[0,s]$, is approximatively $X_0^ue^{-b(s-u)}$ with $b:=\Psi'(0+) \in (-\infty,\infty)$. If the immigration measure $\nu$ has a sufficiently large tail (no log moment), one of those terms is so large that it is the only one to effectively contribute to the sum in \eqref{idform}. Hence, by disregarding the immigration events starting from small initial values, we obtain the approximation:
\begin{center}
$(Y_s,s\geq 0)\approx \big(\underset{0\leq u\leq s} \sup X_0^{u}e^{-b(s-u)},s\geq 0\big)$.
\end{center}
\vspace{-2mm}
For a given time $s>0$ and an immigration time $u\leq s$, there are then two possibilities:  either both terms $X_0^{u}$ and $e^{|b|(s-u)}$ are of the same order (Log case) or  $e^{|b|(s-u)}$ is negligible in front of $X_0^{u}$ (Super-log case). 

In the Log case \eqref{logcaseequiv}, since for any $v>0$, $t\bar{\nu}(e^{tv})\rightarrow c/v$, as $t$ goes to $\infty$, heuristically, the immigration subordinator is ``evolving at an exponential scale".
Hence, when $Y_{st}>0$,
\begin{equation}\label{approxheuristics}\frac{1}{t}\ln Y_{st}\approx \frac{1}{t}\ln \left( \sup_{v\leq st} X_0^{v}e^{-b(st-v)}\right)=\sup_{u_t\leq s} \left(\frac{1}{t}\ln X_0^{tu_t}-b(s-u_t)\right)\underset{t\rightarrow \infty}{\Longrightarrow}\sup_{u\leq s} \left(\xi_u-b(s-u)\right),
\end{equation}
where $u_t:=v/t$ and \begin{equation}\label{cvPPP} \sum_{u>0}\delta_{(u,\xi_u)}:=``\underset{t\rightarrow \infty}{\lim} "\sum_{u_t>0}\delta_{(u_t,\frac{1}{t}\ln X_0^{tu_t})},\end{equation} which turns out to form a PPP whose intensity measure has for tail $\bar{\mu}(v)=c/v$. We see here how the extremal shot noise structure \eqref{maxidform} is arising. 
\smallskip

In the Super-log case \eqref{superlogequivalent}, regardless of whether $b<0$ or $b\geq 0$, the immigration is so strong that the exponential factor $e^{-b(st-v)}$ in \eqref{approxheuristics} is negligible compred to the initial value $X_0^v$. The log function in the renormalisation will have to be replaced by a ``slower" function, for instance $L=\ln\ln$. As a consequence, the exponential term will have no contribution to the limit process and the latter will be a classical extremal process. Notice that the renormalisation function can take here different shapes. This renders the study more involved.
\smallskip

Establishing \eqref{approxheuristics} and \eqref{cvPPP}  rigorously is the approach taken in \cite{IKSANOV2018291} for handling the Log case for GWIs. We will take a different path and work with generators, instead of point processes. 
\smallskip 

It is important to note that the process $Y$ may well hit $0$. To address this and the fact that the $\log$ function is not defined at $0$, we will later use a renormalisation function that is well defined at $0$, for both Log and Super-log cases.


%

\subsection{Extremal shot noise processes}\label{sec:recallESN}
Let $\mu$ be a measure on $(0,\infty)$ such that $\bar{\mu}(x)<\infty$ for all $x>0$. Recall that by definition, an ESN process $M$ based on the measure $\mu$ with slope $-\gamma \in \mathbb{R}$ is a process of the form
\eqref{maxidform}.
We gather in the following lemma some properties that we will need later on. Denote by $C^{1,0}([0,\infty))$ the space of $C^1$ functions vanishing at $\infty$ and whose first-order derivatives are vanishing at $\infty$.
\begin{lemma}[Theorem 1 and Theorem 2 in \cite{FoucartYuanESN23}]\label{lem:coreESN} For any $\gamma\in \mathbb{R}$ and any $\mu$ such that $\int_0 \bar{\mu}(x)\ddr x=\infty$. 
\begin{itemize}
\item[(i)] The process $M$ is Feller.
\item[(ii)] The generator $\mathcal A$  of $M$ has for core the space \begin{align}\mathcal{D}_{\mathrm{ESN}}&:=\{f \in C^{1,0}([0,\infty)): \int_0|f'(x)|\bar{\mu}(x)\ddr x<\infty\}
\label{eqn:dd},
\end{align}
and for any $f\in \mathcal{D}_{\mathrm{ESN}}$ :
\begin{equation}\label{eqn:loggenerator}
    \mathcal Af(x):=\int_x^\infty \Big(f(y)-f(x)\Big)\mu(\ddr y)-\gamma f'(x) \text{ for any } x\geq  0.
    \end{equation}
\end{itemize}
\end{lemma}
\begin{remark}
Note that any function $f$ in $\mathcal{D}_{\mathrm{ESN}}$ satisfies $f'(0)=0$ since by assumption $\int_0\bar{\mu}(x)\ddr x =\infty$. Moreover all functions in $C^{1,0}([0,\infty))$ that are constant near $0$ are in $\mathcal{D}_{\mathrm{ESN}}$. Last, $\mathcal{D}_{\mathrm{ESN}}$ is also a core for the classical extremal process ($\gamma=0$).
\end{remark}
In order to study the convergence of generators, we shall need the following technical lemma ensuring in particular that one can approach any function in $\mathcal{D}_{\mathrm{ESN}}$ with functions in the domain of the CBI.
\begin{lemma}\label{lemmaapprox} For any $f\in \mathcal{D}_{\mathrm{ESN}}$ (recall that it implies $f'(0)=0$), there exists a sequence of
functions $(f_n)$ in $  C^{2}_c([0,\infty))$ and constant near $0$ such that
 \begin{equation}\label{density}
||f-f_n||\underset{n\rightarrow \infty}{\longrightarrow}0 \text{ and } ||\mathcal{A}f_n-\mathcal{A}f||\underset{n\rightarrow \infty}{\longrightarrow}0.
\end{equation}
\end{lemma}
\begin{proof} 
Let $f\in \mathcal{D}_{\mathrm{ESN}}$. Recall that $f'$ is continuous vanishing at $\infty$, $f'(0)=0$ and $\int_0 |f'(y)|\bar{\mu}(y)\ddr y<\infty$. By Stone-Weierstrass theorem, one can construct a sequence of $C^1$ functions $(g_n)$ with compact support on $[0,\infty)$ such that $g_n$ converges towards $f'$ uniformly and in the $L^1$-norm. 

Let $(s_n)$ be a positive sequence going to $0$ such that $s_n\bar{\mu}(1/n)\underset{n\rightarrow \infty}{\longrightarrow} 0$. For such a sequence $(s_n),$ up to working along a subsequence of $(g_{n})$, we can assume that $(g_{n})$ satisfies $\bar{\mu}(1/n+s_n)\|g_{n}-f'\|\underset{n\rightarrow \infty}{\longrightarrow} 0$. 

One can find\footnote{for instance $h_n$ can be defined via a second order polynomial interpolation} a sequence of functions $(h_n)$ so that $\sup_{y\in [1/n,1/n+s_n]} h_n(y) \underset{n\rightarrow \infty}{\longrightarrow} 0$ and the functions $(\bar{g}_n)$ defined via $(h_n)$ below, are in $C_c^1([0,\infty))$. For any $n\geq 1$, set
\[\bar{g}_n(y):=\begin{cases} &0 \text{ if } y\in [0,1/n],\\
&h_n(y) \text{ if } y\in [1/n,1/n+s_n],\\
&g_{n}(y) \text{ if } y\in [1/n+s_n,\infty). \end{cases}
\]
Set for all $x\geq 0$ $$f_n(x):=f(0)+\int_0^x \bar{g}_{n}(y)\ddr y.$$ Note that $f_n$ belongs to $C_c^2([0,\infty))$ and is constant on $[0,1/n]$. Furthermore,
\[ \|f'_{n}-f'\|=\|\bar{g}_{n}-f'\|\leq \underset{0\leq y\leq 1/n}\sup|f'(y)|+\underset{y\in [1/n,1/n+s_n]}\sup|h_n(y)-f'(y)|+\underset{y\in [1/n+s_n,\infty)}\sup|g_n(y)-f'(y)|,\]
and  since $f'(0)=0$, the two first quantities in the upper bound go to $0$, the third one also vanishes since $(g_n)$ converges uniformly towards $f'$. Hence the sequence $(f'_n)$ converges uniformly towards $f'$, as well as  $(f_n)$ towards $f$. 

One has for any $n\geq 1$, using \eqref{eqn:loggenerator}
\begin{align*}
    \|\mathcal{A}f_n-\mathcal{A}f\|\leq |\gamma|\|\bar{g}_{n}-f'\|&+ \int_0^{1/n}|f'(y)|\bar{\mu}(y)\ddr y+\int_{1/n}^{1/n+s_n}|\bar{g}_{n}(y)-f'(y)|\bar{\mu}(y)\ddr y \label{bounddiff}\\
    &+\int_{1/n+s_n}^{1}|g_{n}(y)-f'(y)|\bar{\mu}(y)\ddr y  +\bar{\mu}(1)\int_{1}^{\infty}|g_{n}(y)-f'(y)|\ddr y. \nonumber
\end{align*}
The first term in the upper bound above,  $|\gamma|\|\bar{g}_{n}-f'\|$, converges to $0$. The second term goes to $0$, using the definition of $\mathcal D_{\text{ESN}}$, see \eqref{eqn:dd}. 
For the third one, note that since $\|\bar{g}_{n}-f'\|$ is bounded by some constant $C>0$,
\[\int_{1/n}^{1/n+s_n}|\bar{g}_{n}(y)-f'(y)|\bar{\mu}(y)\ddr y\leq Cs_n \bar{\mu}(1/n) \underset{n\rightarrow \infty}{\longrightarrow} 0.\]
For the fourth term, we have the upper bound 
$$\int_{1/n+s_n}^{1}|g_{n}(y)-f'(y)|\bar{\mu}(y)\ddr y  \leq \bar{\mu}(1/n+s_n)\|g_{n}-f'\|,$$ which goes to $0$.
The last term goes to $0$ as $(g_n)$ converge in $L^1$ towards $f'$. \qed
\end{proof}

\section{Proofs}
\subsection{Proof of Proposition \ref{thm2}}
Let $(c(t), t\geq 0)$ be a positive function and set $h(x):=xc(t)$ for all $x,t\geq 0$.  Plainly, the process  $(c(t)Y_{st},s\geq 0)$ is Feller. Denote by  $\mathcal{A}^{(t)}$ its generator. We have for any function $f$ in $C^2[0,\infty)$, \[\mathcal{A}^{(t)}f(x):=t\mathcal{L}(f\circ h)(h^{-1}(x))=t\mathcal{L}(f\circ h)(x/c(t)).\]
Recall $e_z(x){\blue :=}e^{-zx}$ and note that $e_{z}\circ h(x)=e_{zc(t)}(x)$. 
\\

By \eqref{genCBIonexp}, we have
\[\mathcal{A}^{(t)}e_z(x)=t\left(\frac{x}{c(t)}\Psi(zc(t))-\Phi(zc(t))\right)e_z(x).\]
Recall that by assumption, for some $\alpha, \beta \in (0,1]$, such that $\beta \leq \alpha$, some slowly varying function $\ell$ at $0$ and $d\geq 0, d'>0$,
\begin{equation}\label{eq:equivalencephipsil}
\Phi(x)\underset{x\rightarrow 0}{\sim}d'x^{\beta}\ell(x) \text{ and } \Psi(x)\underset{x\rightarrow 0}{\sim} dx^{1+\alpha}\ell(x).
\end{equation}
Define $c(t):=\Phi^{-1}(1/t)$ for all $t>0$. Plainly, $c(t)\underset{t\rightarrow \infty}{\longrightarrow} 0$ and since $\Phi$ and $\Psi$ are regularly varying at $0$ with index, respectively, $\beta$ and $1+\alpha$, we have
\begin{center}
$t\Phi(zc(t))=\frac{\Phi(zc(t))}{\Phi(c(t))} \underset{t\rightarrow \infty}{\longrightarrow} z^{\beta}=:\bar{\Phi}(z)$
\text{ and }
$\frac{\Psi(zc(t))}{\Psi(c(t))}\underset{t\rightarrow \infty}{\longrightarrow}z^{1+\alpha}$.
\end{center}
Moreover by \eqref{eq:equivalencephipsil}, 
$$t\frac{\Psi(c(t))}{c(t)}=\frac{\Psi(c(t))}{c(t)\Phi(c(t))} \underset{t\rightarrow \infty}{\sim} 
\frac{d}{d'}c(t)^{\alpha-\beta}
\underset{t\rightarrow \infty}{\longrightarrow} \begin{cases} d/d' &\text{ when } \alpha=\beta,\\
0&\text{ when } \alpha>\beta.
\end{cases}$$
We conclude that for all $z\geq 0$,
\begin{center}
$t\frac{\Psi(zc(t))}{c(t)}=\frac{\Psi(zc(t))}{\Psi(c(t))}\frac{\Psi(c(t))}{c(t)}t \underset{t\rightarrow \infty}{\longrightarrow} \bar{\Psi}(z):=\begin{cases} \frac{d}{d}'z^{1+\alpha} &\text{ when } \alpha=\beta,\\
0&\text{ when } \alpha>\beta.\end{cases}$
\end{center}
Let  $\bar{\mathcal{L}}$ be the generator of the CBI$(\bar{\Psi},\bar{\Phi})$, see Section \ref{sec:backgroundCBI} and recall $\bar{\mathcal{L}} e_z$ in \eqref{genCBIonexp}, then
\begin{equation*}\label{upperbound}
||\mathcal{A}^{(t)}e_z-\bar{\mathcal{L}}e_z||\leq \left(t\frac{x}{c(t)}\Psi(zc(t))-\frac{d}{d'}z^{\alpha+1}\right)\times \sup_{x\in [0,\infty)}xe_z(x)+
\left(t\Phi(zc(t))-z^{\beta}\right)\times\sup_{x\in [0,\infty)}e_z(x).
\end{equation*}
Both suprema above are finite and the right-term of the inequality above goes to $0$ as $t$ goes to $\infty$. The convergence of generators is thus uniform on the core $\mathcal{D}_{\mathrm{CBI}}$. Moreover, $c(t)Y_0 \rightarrow 0$ as $t \rightarrow \infty$ and a standard result, see e.g. Kallenberg's book \cite[Theorem 19.25-(i-iv)]{Kallenberg}, ensures then the weak convergence in $\mathbb{D}$:
\[(c(t)Y_{st},s\geq 0)\underset{t\rightarrow \infty}{\Longrightarrow} (X_s,s\geq 0),\]
where $(X_s,s\geq 0)$ is a CBI$(\bar{\Psi},\bar{\Phi})$ started from $0$. The latter shrinks to a subordinator when $\bar{\Psi}\equiv 0$, namely when $\alpha>\beta$. Recall that $c(t)=\Phi^{-1}(1/t)$. Then the proof is finished. \qed

\subsection{Convergence of generators}\label{secproofconvergencegenerator}

We study in this section the convergence of the generators of the renormalised CBIs towards those of extremal and extremal shot noise processes. Recall that by convention $1/c=0$ if $c=\infty$. 

\subsubsection{Preparatory results on the renormalisation function}\label{sec:prepa}
Recall the assumption $\mathbb{H}$ in Theorem \ref{thmdivergencecase} and \eqref{regularvariationassumption}. This is equivalent to the following facts: $x\mapsto 1/\Phi(e^{-x})$ is regularly varying at $\infty$ with some index $\delta\in [0,1]$ and if $\delta=1$ then $x\Phi(e^{-x})\sim \ell(x)$ where $\ell$ is slowly varying and non-decreasing. The regular variation property will not provide good enough properties when we will deal with derivatives. To circumvent this we shall first use an equivalent function $\varphi$ of $\Phi$ such that $r:x\mapsto 1/\varphi(e^{-x})$ is smoothly regularly varying. This will entail that it has derivatives with nice properties. We refer the reader to \cite[Section 1.8.1]{regularvariation} for background on this notion with the definition recalled in the lemma below.  
\begin{lemma}[Equivalent function]\label{eqn:equifun}
If $\mathbb{H}$ holds then there exists a function $\varphi:\R_+\mapsto \R_+$ such that 
\begin{enumerate}
\item $\varphi(1/y)\sim \Phi(1/y)$ as $y\to\infty.$
\item $r:x\mapsto 1/\varphi(e^{-x})$ is smoothly regularly varying at infinity with index $\delta$, that is to say:
$h:x \mapsto \ln r(e^{x})$ is $C^\infty$
and as $x\to\infty$, \begin{equation}\label{eqn:h}h'(x)\to\delta,\quad  h^{(n)}(x)\to0,\quad  n=2,3,\cdots.\end{equation}
\item $\varphi$ is strictly increasing, 
\item $f:x\mapsto x\varphi(e^{-x})$ is non-decreasing.
\end{enumerate}
\end{lemma}
\begin{proof}
If $\delta=0$, define $L(x):=\Phi(e^{-x})$ for $x\geq 1$. Then $L$ is slowly varying at infinity, strictly decreasing, since $\Phi$ increases, and verifies $\lim_{x\to\infty}L(x)=0$. Applying  \cite[Theorem 1.3.3]{regularvariation},  there exists $\widetilde L$ such that  $\widetilde L(x)\sim L(x)$ as $x\to\infty$, and $\widetilde L$ is smoothly slowly varying which implies that 
\begin{equation}\label{eqn:smooth=0}
\widetilde L(x)=\exp\left(a+\int_{1}^x \kappa(s)\frac{\ddr s}{s}\right),\quad x\geq 1,
\end{equation}
where $a\in\R$ and $\kappa(s)\to0$ as $s\to\infty.$  Moreover the construction of $\widetilde{L}$ in the proof of \cite[Theorem 1.3.3]{regularvariation} ensures that if $L$ is strictly decreasing, $\widetilde L$
is also strictly decreasing in $x$. Then $\widetilde L(-\ln x)$ is strictly increasing for $0\leq x\leq e^{-1}$, with convention: $\widetilde L(-\ln 0)=\widetilde L(\infty)=0$. Since $\tilde L$ is slowly varying,  there exists $x_0\geq 1$ such that $x\widetilde L(x)$ is non-decreasing for $x\geq x_0.$
Then we can set $\varphi(x)=\widetilde L(-\ln x)$ on $[0,e^{-x_0}]$. Note that $x\varphi(e^{-x})=x\widetilde L(x)$. Then all the four conditions are satisfied with $\varphi(x)$ for $x\in[0,e^{-x_0}]$. It then suffices to extend $\varphi$ to $(e^{-x_0},\infty)$ with necessary smoothness properties and monotonicity. 
The proof for the case of $\delta=0$ is finished.  

If $0<\delta<1$, define $L$ such that  $\Phi(e^{-x})=x^{-\delta}L(x)$ for $x\geq 1$. Then $L$ is slowly varying at infinity. Using the same arguments as above, there exists $\overline L$ such that $\overline L(x)\sim L(x)$ as $x\to\infty$, and $\overline L$ is smoothly slowly varying which implies that \begin{equation}\label{eqn:smooth>0}
\overline L(x)=\exp\left(b+\int_{1}^x \zeta(s)\frac{\ddr s}{s}\right),\quad x\geq 1,
\end{equation}
where $b\in\R$ and $\zeta(s)\to0$ as $s\to\infty.$ Similarly, we can set $\varphi(x):=(-\ln x)^{-\delta}\overline L(-\ln x)$ for $x$ close enough to $0$ and extend it to the whole positive half line with necessary smoothness properties and monotonicity so that $\varphi$ satisfies all the four conditions. 
Then the proof for the case of $0<\delta<1$ is completed. 


If $\delta=1$, we recall that we \textit{assume} that $x\Phi(e^{-x})\sim \ell(x)$ as $x$ goes to $\infty$ with $\ell$ slowly varying and \textit{ non-decreasing}. By the same arguments as above in the case $\delta=0$ (i.e.\ use \cite[Theorem 1.3.3]{regularvariation} and its proof), there exists $\widetilde \ell$ such that $\widetilde \ell(x)\sim \ell(x)$ as $x\to\infty$, $\widetilde \ell$ is smoothly slowly varying and 
\begin{equation}\label{eqn:smoothdelta=1}
\widetilde \ell(x)=\exp\left(\rho+\int_{1}^x \eta(s)\frac{\ddr s}{s}\right),\quad x\geq 1,
\end{equation}
where $\widetilde \ell$ is non-decreasing, $\rho\in\R$ and $\eta(s)\to0$ as $s\to\infty.$ 
Then for $x$ small enough, the function $\varphi(x):=(-\ln x)^{-1}\widetilde \ell(-\ln x)$ is strictly increasing. 
It suffices to extend $\varphi$ to the whole positive half line with necessary smooth properties and monotonicity such that $\varphi$ satisfies all the first three conditions. The last condition is also met since $x\varphi(e^{-x})=\widetilde \ell(x)$.  Then the case for $\delta=1$ is proved. \qed
\end{proof}

We shall now work with the function $\varphi$ as introduced in Lemma \ref{eqn:equifun} instead of $\Phi$. We will go back to $\Phi$ at the end of the proof. From now on, Assumption $\mathbb{H}$ is in force. Notice that it entails that
\begin{equation*}
x\varphi(e^{-x})\underset{x\rightarrow \infty}{\longrightarrow}c \in (0,\infty].
\end{equation*}
The following is a key lemma.
\begin{lemma}\label{lem:twoderivezero} Let $\varphi$ be an equivalent function as in Lemma \ref{eqn:equifun}. One has 
\begin{equation}\label{eqn:firstconv}\lim_{u\to \infty}\frac{\varphi'(1/u)}{u\varphi^2(1/u)}=1/c
\end{equation}
and \begin{equation}\label{eqn:secondconv}\lim_{u\to \infty}\frac{\varphi''(1/u)}{u^2\varphi^2(1/u)}=-1/c.\end{equation}
\end{lemma}
\begin{proof} Recall that in the log case $\delta=1$. For \eqref{eqn:firstconv}, one has:
\begin{align*}
\frac{\ddr }{\ddr x}\frac{1}{\varphi(e^{-x})}=
\frac{\varphi'(e^{-x})}{\varphi^2(e^{-x})}e^{-x}=h'(\ln x)x^{-1}e^{h(\ln x)}=(\delta +o(1)) e^{h(\ln x)}x^{-1}=\frac{1}{x\varphi(e^{-x})}(\delta +o(1)) \underset{x\rightarrow \infty}{\longrightarrow} \frac{1}{c}. 
\end{align*}
For \eqref{eqn:secondconv}, we have, as $x$ goes to $\infty$: 
$$\frac{\ddr^2 }{\ddr x^2}\frac{1}{\varphi(e^{-x})}=-\frac{e^{-x}\varphi'(e^{-x})}{\varphi^2(e^{-x})}-\frac{e^{-2x}\varphi''(e^{-x})}{\varphi^2(e^{-x})}+\frac{2e^{-2x}(\varphi'(e^{-x}))^2}{\varphi^3(e^{-x})}=-\frac{1}{c}-\frac{e^{-2x}\varphi''(e^{-x})}{\varphi^2(e^{-x})}+o(1).$$
The last equality is due to \eqref{eqn:firstconv} and $\varphi(e^{-x})\to 0$ as $x\to\infty$.
Next note that: 
$$\frac{\ddr^2 }{\ddr x^2}\frac{1}{\varphi(e^{-x})}=\frac{\ddr}{\ddr x}h'(\ln x)x^{-1}e^{h(\ln x)}=h''(\ln x)x^{-2}e^{h(\ln x)}\to0 \text{ as } x \to \infty,$$
where the convergence is due to the fact that  $h''(\ln x)$ converges to a constant as $x\to\infty$, see \eqref{eqn:h}, and $x^{-2}e^{h(\ln x)}=\frac{1}{x^2\varphi(e^{-x})}\sim \frac{1}{x^2\Phi(e^{-x})}\to0$ as $x$ goes to $\infty$. The above two displays yield \eqref{eqn:secondconv} and the proof is finished. \qed 
\end{proof}

\subsubsection{Preparatory results on the prelimiting generators}
Let $g:y\mapsto g(y)$ be a continuous strictly increasing function tending to $\infty$ as $y$ goes to $\infty$.
Fix $t>0$, and let $x:=g(y)\geq 0$. 
The process $(g(Y_{st}),s\geq 0)$ has for semigroup
\[\mathbb{E}[f(g(Y_{st}))\,|\, Y_{st}=g^{-1}(x)]=: P^{(t)}_sf(x)=P_{st}(f\circ g)(g^{-1}(x)),\]
where $(P_{s})$ denotes the semigroup of $(Y_s,s\geq 0)$. Since the CBI process is Feller, for any $s\geq 0$, $P_s$ maps $C_0([0,\infty))$ into itself and for any $f\in C_0([0,\infty))$, $x\mapsto P^{(t)}_sf(x)$ is continuous on $[0,\infty)$ and vanishing at $\infty$. The process $(g(Y_{st}),s\geq 0)$ is therefore Feller.

Let $\mathcal A^{(t)}$ be the generator of $(g(Y_{st}),s\geq 0)$. That is, 
$$\lim_{h\to0}\frac{P_{st}(f\circ g)(g^{-1}(x))-f(x)}{h}=\mathcal A^{(t)}f(x)=t\Big(\mathcal L (f\circ g)\Big) (g^{-1}(x)),$$
uniformly in $x\geq 0$ for any $f$ in the domain of $\mathcal A^{(t)}$. 

Recall that $\varphi$ is strictly increasing. For any $t>0$, define 
\begin{equation}\label{defg} g(y):=\frac{1}{t\varphi(1/y)} \text{ for all } y\geq 0 \text{ and its inverse }
g^{-1}(x)=\frac{1}{\varphi^{-1}(1/xt)} 
 \text{ for all }x\geq 0.\end{equation}
We shall renormalise the CBI process with the function $g$. Notice that $g$ depends on $t$ and for $t>0$ fixed, $g$ is strictly increasing. Furthermore, $y=g^{-1}(x)$ goes to infinity as $t\to\infty$ with $x$ fixed. We shall treat both Log and Super-log cases at once. However in the Log case, many of the computations to come can be simplified using the explicit form $\varphi(x)=\frac{c}{\ln(1+1/x)}$. Indeed in this setting, we can set for any $t\geq 0, x,y\geq 0$
\[g(y):=\frac{\ln(1+y)}{ct} \text{ and } g^{-1}(x):=e^{ctx}-1.\]

The following lemma is a key as it explains how the measure $\mu(\ddr x):=\ddr x/x^2$ arises at the limit. 
\begin{lemma}\label{lem: fastjumpdiff} Assume  $\mathbb{H}$. Recall the measure $\nu$ from \eqref{eqpsi}.
Let $c_0>0$. 
For any $x\geq 0$, $v\geq x+c_0$, we have uniformly in $v,x$
\begin{equation}\label{tailconvergence}
t\overline \nu\big(g^{-1}(v)-{g^{-1}(x)}\big)\xrightarrow[t\to\infty]{}\bar{\mu}(v):=1/v.
\end{equation}
\end{lemma}
\begin{proof}
The regular variation of $1/\Phi(e^{-x})$ implies that $\Phi$ is slowly varying near zero. Then necessarily $\Phi'(0)=\int_0^{\infty}\bar{\nu}(u)\ddr u=\infty$. Hence, by Fubini-Tonelli theorem
\begin{equation}\label{fubinitonelli} \frac{\Phi(x)}{x}=\beta+\int_0^{\infty}e^{-ux}\bar{\nu}(u)\ddr u \underset{x\rightarrow 0}{\sim} \int_0^{\infty}e^{-ux}\bar{\nu}(u)\ddr u.\end{equation}
By applying Karamata Tauberian theorem, see \cite[Theorem 1.7.1]{regularvariation}, one has $\int^z\bar{\nu}(u)\ddr u \underset{z\rightarrow \infty}{\sim} z\Phi(1/z)$. Since  $u\mapsto \bar{\nu}(u)$ is monotone, we deduce by applying the monotone density theorem, see \cite[Theorem 1.7.2]{regularvariation}, that
\begin{equation}\label{equivPhinu}\bar{\nu}(u)\underset{u\rightarrow \infty}{\sim} \Phi(1/u).\end{equation} 
Recall $g^{-1}(x)=\frac{1}{\varphi^{-1}(1/xt)}$. Since $\Phi(1/y)\sim\varphi(1/y)$ as $y$ goes to $\infty$, we have 
\begin{equation}\label{eqn:g}t\bar{\nu}(g^{-1}(x))=t\bar{\nu}(1/\varphi^{-1}(1/xt))
\underset{t\rightarrow \infty}{\longrightarrow}1/x.\end{equation}
Since $\bar \nu$ is decreasing and $g^{-1}$ is increasing in $x$, the map $x\mapsto \bar{\nu}(g^{-1}(x))$ is decreasing and the convergence in \eqref{eqn:g} holds uniformly in $x\geq \epsilon$ for some fixed $\epsilon>0$, see e.g. \cite[Item 127 pages 81 and 270]{zbMATH01257769}. 

Next, we are going to show the following uniform convergence: for any $c_0>0$ \begin{equation}\label{eqn:A}\underset{t\rightarrow \infty}{\lim}\, \underset{v\geq x+c_0, x\geq 0}{\inf}\frac{g^{-1}(v)
}{g^{-1}(x)}
=\underset{t\rightarrow \infty}{\lim}\, \underset{x\geq 0}{\inf}\frac{1/\varphi^{-1}(1/(x+c_0)t)
}{1/\varphi^{-1}(1/xt)
}=\infty.\end{equation}
Note that the first equality uses the definition of $g^{-1}$, see \eqref{defg}, and the monotonicity of $g^{-1}$. So we only have to prove the second equality. We first prove a partial result as follows: for any $\epsilon>0,$
\begin{equation}\label{eqn:xepsilon0}\underset{t\rightarrow \infty}{\lim}\, \underset{v\geq x+c_0, x\geq \epsilon>0}{\inf}\frac{1/\varphi^{-1}(1/(x+c_0)t)
}{1/\varphi^{-1}(1/xt)
}=\infty.\end{equation}
Recall that $\varphi$ is strictly increasing, we have for all $x>0$
\[\varphi^{-1}(1/x)=\exp\Big(-xf(-\ln\varphi^{-1}(1/x))\Big),\]
where $f:x\mapsto x\varphi(e^{-x})$. Then, uniformly in $x\geq \epsilon$,
\begin{align*}\frac{1/\varphi^{-1}(1/(x+c_0)t)
}{1/\varphi^{-1}(1/xt)
}&=\exp\Big((x+c_0)tf(-\ln\varphi^{-1}(1/(x+c_0)t))-xtf(-\ln\varphi^{-1}(1/xt))\Big)\\
&\geq \exp\Big(c_0tf(-\ln\varphi^{-1}(1/(x+c_0)t))\Big) \xrightarrow[t\to\infty]{} \infty.\end{align*} Here the inequality uses that  $x\mapsto -\ln\varphi^{-1}(1/x)$ is strictly increasing to infinity and $x\mapsto f(x)=x\varphi(e^{-x})$ is non-decreasing, see Lemma \ref{eqn:equifun}. Then \eqref{eqn:xepsilon0} is proved.

In fact, the assumption $x\geq \epsilon>0$ in \eqref{eqn:xepsilon0} can be replaced by $x\geq 0$. Setting $z=x+c_0/2$ and using the fact that $\varphi^{-1}$ is increasing, we have :
$$\frac{g^{-1}((x+c_0)t)}{g^{-1}(xt)}= \frac{\varphi^{-1}(\frac{1}{xt})}{\varphi^{-1}(\frac{1}{(x+c_0)t})}\geq \frac{\varphi^{-1}(\frac{1}{(x+c_0/2)t})}{\varphi^{-1}(\frac{1}{(x+c_0)t})}=\frac{\varphi^{-1}(\frac{1}{zt})}{\varphi^{-1}(\frac{1}{(z+c_0/2)t})}=\frac{g^{-1}((z+c_0/2)t)}{g^{-1}(zt)}, \quad z\geq c_0/2.$$
Using \eqref{eqn:xepsilon0}, the last term goes to $\infty$ uniformly in $z\geq \epsilon:=c_0/2$. This allows us to conclude that the first term in the above display also goes to $\infty$ uniformly for $x\geq 0$. Therefore,  \eqref{eqn:A} is proved. 

Finally, we are ready to prove \eqref{tailconvergence}. Since $\bar\nu$ is slowly varying, the uniform convergence theorem for slowly varying functions, see \cite[Theorem 1.2.1]{regularvariation}, together with \eqref{eqn:g} and \eqref{eqn:A}, entail
\begin{equation}\label{eqn:g/g}
t\overline \nu\big(g^{-1}(v)-{g^{-1}(x)}\big)=t\overline \nu\left(g^{-1}(v)\left(1-\frac{g^{-1}(x)}{g^{-1}(v)}\right)\right)\xrightarrow[t\to\infty]{}1/v.
\end{equation}
The convergence holds uniformly in $v\geq x+c_0$ and $x\geq 0$. Then the proof is finished. \qed

\end{proof}

\bigskip


Recall $g(y)=\frac{1}{t\varphi(1/y)}$, see \eqref{defg}, and for any fixed $x$, set $y=g^{-1}(x)$. Let $t>0$ be fixed. The generator of the prelimiting process $(g(Y_{st}),s\geq 0)$ takes the following form  
\begin{align*}\mathcal A^{(t)}f(x)&:=t\mathcal L(f\circ g) (y)\\
&=ty\mathrm{L}^{\Psi}(f\circ g)(y)+t\mathrm{L}^{\Phi}(f\circ g)(y)\\
&=I_1+I_2+I_3+I_4,
\end{align*}
with respectively $I_1$ the drift term, $I_2$ the immigration jump term, $I_3$ the diffusion term and $I_4$ the branching jump term:
\begin{align}
  I_1&:=t(\beta-by)(f\circ g)'(y)    \label{I1}\\
  I_2&:=t\int_0^\infty \Big(f\circ g(y+u)-f\circ g (y)\Big)\nu(\ddr u)\label{I2}\\
  I_3&:=t\frac{\sigma^2}{2}y(f\circ g)''(y) \label{I3}\\
  I_4&:=ty\int_0^\infty \Big(f\circ g(y+u)-f\circ g (y)-u(f\circ g)'(y)\Big)\pi(\ddr u).\label{I4}
\end{align}
Recall $\mathcal{A}$ the generator of the ESN$(b/c,\mu)$ process $M$, see \eqref{eqn:logconv} with $\mu(\ddr y):=\ddr y/y^2$. Note that $\bar{\mu}(y)=1/y$ so that $\int_0 \bar{\mu}(y)\ddr y=\infty$ and by Lemma \ref{lem:coreESN}, $\mathcal{D}_{\mathrm{ESN}}$ defined in \eqref{eqn:dd} is a core. For any $f\in \mathcal{D}_{\mathrm{ESN}}$
and $x\geq  0$, an application of Fubini-Lebesgue's theorem yields
\begin{align*}\mathcal{A}f(x)&:=\int_x^\infty \Big(f(y)-f(x)\Big)\frac{\ddr y}{y^2}-\frac{b}{c}f'(x)\\
&=\int_x^\infty f'(z)\frac{\ddr z}{z}-\frac{b}{c}f'(x). 
\end{align*}
We shall next establish the following lemma.  
\begin{lemma}\label{lem:ata}For any $f\in C_c^2[0,\infty)$ \textit{constant near zero}, 
\begin{equation}\label{convergenceofgenerators}\mathcal{A}^{(t)}f\rightarrow \mathcal{A}f \text{ uniformly as } t \text{ goes to }\infty. \end{equation}
\end{lemma}
 Notice that any $f\in C_c^2[0,\infty)$, both $f$ and $f\circ g$ belong to the domain of the CBI $(Y_s,s\geq 0)$ and $f$ to that of the ESN $(M_s,s\geq 0)$. As we can see by the form of the generator $\mathcal A$, $I_3$ and $I_4$ will have no contribution at the limit. The proof of the above lemma is given by the two subsequent sections. 

\subsubsection{Proof of Lemma \ref{lem:ata}: uniform convergence on $[\epsilon,\infty)$}\label{sec:convawayfrom0}
We first show the uniform convergence of generators \eqref{convergenceofgenerators} for $x\geq \epsilon$. The assumption of $f$ being constant near $0$ will not play a role in this part.  We will prove it in four steps, each dealing with one term above. \\

\noindent \underline{Step 1}:  drift term $I_1$. 
Note that 
\begin{equation}\label{eqn:drift}(\beta-by)(f\circ g)'(y)=(\beta-bg^{-1}(x))(f'\circ g)(g^{-1}(x))\times g'(g^{-1}(x))\end{equation}

Plugging in $y=g^{-1}(x)$ to \eqref{eqn:drift} and using that
$g'(y)=\frac{1}{y^2}\frac{\varphi'(1/y)}{\varphi^2(1/y)}t^{-1}$, we have 
\begin{align}\label{eqn:fastdrift}(\beta-by)(f\circ g)'(y)&=(\beta-bg^{-1}(x))(f'\circ g)(g^{-1}(x))\times g'(g^{-1}(x))\nonumber\\
&=(\beta-by)\frac{1}{y^2}\frac{\varphi'(1/y)}{\varphi^2(1/y)}t^{-1}f'(x).
\end{align}
Using \eqref{eqn:fastdrift} and \eqref{eqn:firstconv},  \eqref{eqn:secondconv}
from \Cref{lem:twoderivezero}, we get the pointwise convergence for $x\geq \epsilon$ fixed \begin{equation}\label{pointwisedrift} I_1=t(\beta-by)(f\circ g)'(y)\xrightarrow[t\to\infty]{} \frac{b}{c}f'(x).
\end{equation}
For the uniform convergence when $x\geq \epsilon$, we shall use the following fact: since $y=g^{-1}(x)$ is increasing in $x$, 
\begin{equation}\label{uniform bymonotonicity}
\text{when } x\geq \epsilon,  y=g^{-1}(x)\geq g^{-1}(\epsilon)\underset{t\rightarrow \infty}{\longrightarrow} \infty.
\end{equation}
Therefore
\begin{align}\label{uniform bymonotonicity2}
\sup_{x\geq \epsilon}\left\lvert t(\beta-by)(f\circ g)'(y)-\frac{b}{c}f'(x)\right\lvert& 
=\sup_{x\geq \epsilon}\left\lvert
(\beta-by)\frac{1}{y^2}\frac{\varphi'(1/y)}{\varphi^2(1/y)}-\frac{b}{c}\right\lvert\times |f'(x)| \nonumber \\
&\leq  \sup_{u\geq g^{-1}(\epsilon)}\left\lvert
(\beta-bu)\frac{1}{u^2}\frac{\varphi'(1/u)}{\varphi^2(1/u)}-\frac{b}{c}\right\lvert\times ||f'||, 
\end{align}
the last upper bound goes to $0$ as $t$ goes to $\infty$ by Lemma \ref{lem:twoderivezero}. The convergence is thus uniform for all $x\geq \epsilon$. \qed 
\\
\noindent  \underline{Step 2}: immigration jump term $I_2$. By applying Fubini-Lebesgue's theorem, we can rewrite it as follows:

\begin{align*}
I_2=t\int_0^\infty \Big(f\circ g(y+u)-f\circ g (y)\Big)\nu(\ddr u)&\,=t\int_y^\infty f'(g(u))g'(u)\bar\nu(u-y)\ddr u\nonumber\\
  &\,=t\int_{x}^\infty f'(z)\bar\nu(g^{-1}(z)-g^{-1}(x))\ddr z.\label{eqn:barnuintegral}
\end{align*} 
Next, we are going to show that the last integral term converges uniformly in $x\geq \epsilon$ towards  
$\int_x^\infty f'(v)v^{-1}\ddr v$. We first show the convergence on an interval of the form $[x+c_0,\infty)$ with $c_0>0, x\geq \epsilon>0$. Since $f'$ is integrable near $\infty$, 
by using Lemma \ref{lem: fastjumpdiff} and Lebesgue theorem, we get
\begin{align*} 
    \sup_{x\geq \epsilon}\Big\lvert\int_{x+c_0}^\infty f'(z)&\left(t\bar\nu(g^{-1}(z)-g^{-1}(x))\ddr z- 1/z\right)\ddr z \Big\lvert\\
    &\leq  \int_{0}^\infty |f'(z)|\sup_{x\geq \epsilon}\left \lvert t\bar\nu(g^{-1}(z)-g^{-1}(x))- 1/z  \right \lvert \mathbbm{1}_{\{z\geq x+c_0\}}\ddr z \underset{t\to\infty}{\to} 0.
\end{align*}
Next we deal with the integral from $x$ to $x+c_0$ and show that this integral is arbitrarily small uniformly with respect to $x\in [\epsilon,\infty)$, by choosing $c_0$ small. We note that 
\begin{align*}
&t\int_x^{x+c_0} f'(z)\bar\nu(g^{-1}(z)-g^{-1}(x))\ddr z\\
&=t\int_0^{g^{-1}(x+c_0)-g^{-1}(x)}(f\circ g(g^{-1}(x)+u)-f(x))\nu(\ddr u)\\
&= t\int_0^{g^{-1}(x)}(f\circ g(g^{-1}(x)+u)-f(x))\nu(\ddr u)+t\int_{g^{-1}(x)}^{g^{-1}(x+c_0)-g^{-1}(x)}(f\circ g(g^{-1}(x)+u)-f(x))\nu(\ddr u)\\
&=:J_{1}+J_2.\end{align*}
By Lemma \ref{lem: fastjumpdiff}, note that $J_2\leq 2\|f\|t(\bar\nu(g^{-1}(x))-\bar\nu(g^{-1}(x+c_0)))\underset{t\rightarrow \infty}{\rightarrow} 2\|f\|\left(\frac{1}{x}-\frac{1}{x+c_0}\right)$ which is, for small $c_0$,  arbitrarily small uniformly in $x\geq \epsilon$. For $J_1,$ by applying the mean value theorem to the function $f\circ g$ in the integrand, one has
\[f\circ g(g^{-1}(x)+u)-f\circ g(g^{-1}(x))=f'(g(\theta))g'(\theta)u,\]
for some $\theta\in [g^{-1}(x),g^{-1}(x)+u]\subset [g^{-1}(x),2g^{-1}(x)]$ and
\begin{align*}J_{1}&=t\int_0^{g^{-1}(x)}f'(g(\theta))g'(\theta)u\nu(\ddr u)\\
&=t\int_0^{C}f'(g(\theta))g'(\theta)u\nu(\ddr u)+t\int_C^{g^{-1}(x)}f'(g(\theta))g'(\theta)u\nu(\ddr u)\\
&\leq \|f'\|t \sup_{g^{-1}(x)\leq \theta \leq 2g^{-1}(x)}\!\! g'(\theta)\int_0^C u\nu(\ddr u)+\|f'\|t\sup_{g^{-1}(x)\leq \theta\leq 2g^{-1}(x)}g'(\theta)g^{-1}(x)\bar\nu(C)
\end{align*}
for any $C>0$. We now explain why $J_1$ is arbitrarily small uniformly in $x\geq \epsilon$. Recall $tg(x)=\frac{1}{\varphi(1/x)}$. For the first term on the right-hand side in the last line of the above display, with fixed $C$, one has by Lemma \ref{lem:twoderivezero}, see \eqref{eqn:secondconv}.
\begin{equation}\label{convergencetozeroforu}
tg'(u)=-\frac{\varphi'(1/u)}{\varphi(1/u)^2 u^2}\underset{u\rightarrow \infty}{\longrightarrow} 0,\end{equation} 
Since $tg'$ does not depend on $t$ and  $y=g^{-1}(x)$ goes to $\infty$ as $t$ goes to $\infty$, $$tg'(g^{-1}(x))\underset{t\rightarrow \infty}{\longrightarrow} 0.$$
Moreover, using the representation $tg(u)=\frac{1}{\varphi(1/u)}=e^{h(\ln \ln u)}$ with the fact that $\underset{u\rightarrow \infty}{\lim} h'(u)=\delta$ and that $g$ is slowly varying, we see that $tg'(u)=h'(\ln \ln u) g(u) (\ln\ln)'(u)$ is regularly varying at $\infty$. 
Therefore, by the uniform convergence theorem, see \cite[Theorem 1.2.1]{regularvariation}, we have
\begin{equation}\label{supequiv}
\sup_{u\leq \theta\leq 2u}{tg'(\theta)}\underset{u\rightarrow \infty}{\sim} tg'(u).
\end{equation}
Recall \eqref{uniform bymonotonicity}, in the same fashion as in \eqref{uniform bymonotonicity2}, one has
\[\sup_{x\geq \epsilon} \sup_{g^{-1}(x)\leq \theta\leq 2g^{-1}(x)}{tg'(\theta)}\leq \sup_{u\geq g^{-1}(\epsilon)} \sup_{u\leq \theta\leq 2u}{tg'(\theta)}\]
which goes to $0$ as $t$ goes to $\infty$ by \eqref{supequiv} and \eqref{convergencetozeroforu}.

For the second term,  note that by \eqref{supequiv}
$$\sup_{g^{-1}(x)\leq \theta\leq 2g^{-1}(x)}tg'(\theta)g^{-1}(x) \underset{t\rightarrow \infty}{\sim} tg'(g^{-1}(x))g^{-1}(x),$$
and by Lemma \ref{lem:twoderivezero}, see \eqref{eqn:firstconv}, $tg'(g^{-1}(x))g^{-1}(x)$ is bounded above uniformly in $x\geq \epsilon$. Since $\bar{\nu}(C)$ goes to $0$ as $C$ goes to $\infty$, the second part is arbitrarily small uniformly for $x\geq \epsilon$. 

In conclusion, we have shown that $I_2$ converges uniformly in $x\geq \epsilon$ to $\int_{x}^\infty f'(z)\frac{\ddr z}{z}$ as $t\to\infty$. 

\qed

\noindent \underline{Step 3}:  diffusion term $I_3$. For any $u$,
\begin{align}\label{eqn:doubleprime}(f\circ g)''(u)&=(f''\circ g)(u)\times (g'(u))^2+(f'\circ g)(u)\times g''(u).\end{align}
Plugging in \eqref{eqn:doubleprime} $u=y$ and $g(y)=\frac{1}{t\varphi(1/y)}$ leads to $$(f\circ g)''(y)=f''\left(\frac{1}{t\varphi(1/y)}\right)\frac{\left(\varphi'(1/y)\right)^2}{\varphi^4(1/y)}y^{-4}t^{-2}-f'\left(\frac{1}{t\varphi(1/y)}\right)\cdot \left(\frac{2}{y^3}\frac{\varphi'(1/y)}{\varphi^2(1/y)}+\frac{1}{y^4}\frac{\varphi''(1/y)}{\varphi^2(1/y)}+\frac{2}{y^4}\frac{(\varphi'(1/y))^2}{\varphi^3(1/y)}\right)t^{-1}.$$
Plugging in $y=g^{-1}(x)$, we obtain \begin{equation}\label{eqn:fastdiffusion}(f\circ g)''(y)=f''\left(x\right)\frac{\left(\varphi'(1/y)\right)^2}{\varphi^4(1/y)}y^{-4}t^{-2}-f'\left(x\right)\cdot \left(\frac{2}{y^3}\frac{\varphi'(1/y)}{\varphi^2(1/y)}+\frac{1}{y^4}\frac{\varphi''(1/y)}{\varphi^2(1/y)}+\frac{2}{y^4}\frac{(\varphi'(1/y))^2}{\varphi^3(1/y)}\right)t^{-1}.\end{equation}
From \eqref{eqn:fastdiffusion}, by applying \eqref{eqn:firstconv} and \eqref{eqn:secondconv}, we can see that \begin{equation}\label{step3convergence}
    ty(f\circ g)''(y)\xrightarrow[t\to\infty]{}0.
\end{equation}
Indeed all terms in $ty(f\circ g)''(y)$  will vanish when $t$ goes to $\infty$, as they are of order at most $\frac{1}{y^2}$. The above convergence is uniform for all $x\geq \epsilon$ since $\|f''\|<\infty$ and $\|f'\|<\infty$. Since \\ 

\noindent \underline{Step 4}: branching jump term $I_3$.

Note that 
\begin{align}
ty\int_0^\infty &\Big(f\circ g(y+u)-f\circ g (y)-u(f\circ g)'(y)\Big)\pi(\ddr u) \label{eqn:step2aim} \\
&=ty\int_0^C \Big(f\circ g(y+u)-f\circ g (y)-u(f\circ g)'(y)\Big)\pi(\ddr u) \label{smalljumps} \\
 & +ty\int_C^\infty \Big(f\circ g(y+u)-f\circ g (y)\Big)\pi(\ddr u) - tya_C(f\circ g)'(y), \label{bigjumps}
\end{align} 
with $a_C:=\int_C^{\infty}u\pi(\ddr u)$. We show that \eqref{eqn:step2aim} converges to zero uniformly in $x\geq \epsilon$. To this purpose, it suffices to prove that 
\eqref{smalljumps} converges uniformly to zero for any given $C$, and \eqref{bigjumps} can be arbitrarily small in absolute value by choosing $C$ large enough, uniformly in $x\geq \epsilon$ for large $t$. 

For the first term \eqref{smalljumps}, notice that by \eqref{step3convergence},
$$t(y+uv)\left((f\circ g)''(y+uv)\right) \underset{t\rightarrow \infty}{\rightarrow} 0,$$
and if $0<uv\leq C$, we get
$$ty\left((f\circ g)''(y+uv)\right) \underset{t\rightarrow \infty}{\rightarrow} 0,$$
also uniformly in $x\geq \epsilon$. Recall that $\int_0^C u^2\pi(\ddr u)<\infty$, by the dominated convergence theorem we have for any $C>0$ fixed
\begin{align}
&ty\int_0^C \Big(f\circ g(y+u)-f\circ g (y)-u(f\circ g)'(y)\Big)\pi(\ddr u)\nonumber \\
=&\int_0^C u^2\pi(\ddr u)\int_0^1 ty\left((f\circ g)''(y+uv)\right)(1-v)\ddr v  
\underset{t\rightarrow \infty}{\rightarrow} 0 \text{ uniformly in } x\geq \epsilon. \label{expressionsmalljumpsPsi0}
\end{align}

Now we deal with \eqref{bigjumps}.  Recall that we have shown in Step 1 that $ty(f\circ g)(y)\underset{t\rightarrow \infty}{\rightarrow} f'(x)$ uniformly in $x\geq \epsilon$. Hence the second term in \eqref{bigjumps} converges towards $-f'(x)\int_C^\infty u\pi(\ddr u)$ which can be arbitrarily small in absolute value  if we take $C$ large enough, uniformly in $x\geq \epsilon$.
That implies that for any $\eta>0$, there exists $t_\eta, C_\eta$ such that
\begin{equation}\label{eqn:step2part1}\sup_{t\geq t_\eta, C\geq C_\eta}\left|- tya_C(f\circ g)'(y)\right|\leq \eta.\end{equation}
As for the integral term in 
\eqref{bigjumps},  note that $$(f\circ g)'(v+y)=f'(g(v+y))g'(v+y),$$ and one can re-express the integral by 
\begin{align*}\label{controlC}
    &\left|ty\int_C^\infty \Big(f\circ g(y+u)-f\circ g (y)\Big)\pi(\ddr u)\nonumber\right| \\
    &=\left|ty\int_C^\infty \pi(\ddr u) \int_0^u f'(g(v+y))g'(v+y)\ddr v \right|\\
    &\leq \|f'\|C_1a_C,
\end{align*}
where $C_1=\sup_{ y\geq 0, v\geq 0}|tyg'(v+y)|\leq \sup_{y\geq 0}|tyg'(y)|<\infty$ by 
\eqref{eqn:firstconv}. Therefore by choosing $C$ large enough, the first term in \eqref{bigjumps} is arbitrarily small in absolute value uniformly for all $x\geq 0$. Together with \eqref{eqn:step2part1} and \eqref{expressionsmalljumpsPsi0}, we conclude that \eqref{eqn:step2aim} converges to zero uniformly in $x\geq \epsilon.$ Then the proof for Step 4 is finished. \qed

\subsubsection{Proof of Lemma \ref{lem:ata}: uniform convergence on $[0,\epsilon]$}\label{secconv0}   
We have shown in the previous section that 
\begin{equation}\label{supepsilontoinfinity}
\sup_{x\in [\epsilon,\infty)}|\mathcal{A}^{(t)}f(x)-\mathcal{A}f(x)|\underset{t\rightarrow \infty}{\longrightarrow} 0.\end{equation}
We study now the supremum over $x\in [0,\epsilon]$ and establish that for any function $f$ in $C^{2}_c([0,\infty))$ which is \textit{constant near} $0$, one has:
\begin{equation}\label{supepsilonto0}
\sup_{x\in [0,\epsilon]}|\mathcal{A}^{(t)}f(x)-\mathcal{A}f(x)|\underset{t\rightarrow \infty}{\longrightarrow} 0
\end{equation}
It is equivalent to showing that for any sequence $x_t\in [0,\epsilon]$ such that $x_t\underset{t\rightarrow\infty}{\longrightarrow} 0$, we have 
\begin{equation}
    \label{subsequence}
|\mathcal{A}^{(t)}f(x_t)-\mathcal{A}f(x_t)|\underset{t\rightarrow \infty}{\longrightarrow} 0.
\end{equation}
Indeed \eqref{supepsilonto0} plainly entails \eqref{subsequence}. If \eqref{subsequence} is true, then \eqref{supepsilonto0} must be true. In fact, if \eqref{supepsilonto0} does not hold, then \[\underset{t\rightarrow \infty}{\liminf}\sup_{x\in [0,\epsilon]}|\mathcal{A}^{(t)}f(x)-\mathcal{A}f(x)|>0.\]
Hence there exists a sequence $(y_t)$ in $[0,\epsilon]$ such that 
\[\underset{t\rightarrow \infty}{\lim} |\mathcal{A}^{(t)}f(y_t)-\mathcal{A}f(y_t)|>0.\]
If $\underset{t\rightarrow \infty}{\lim} y_t=0$ this contradicts \eqref{subsequence} and if $\underset{t\rightarrow \infty}{\limsup} y_t>0$, then by working along any convergent subsequence of $(y_t)$, we would get that the convergence \eqref{supepsilontoinfinity} previously obtained does not hold.

To prove \eqref{subsequence}, we first observe that
\[|\mathcal{A}^{(t)}f(x_t)-\mathcal{A}f(x_t)|\leq |\mathcal{A}^{(t)}f(x_t)-\mathcal{A}f(0)|+|\mathcal{A}f(0)-\mathcal{A}f(x_t)|.\]
Clearly the second term on the right-hand side converges towards $0$ as $t$ goes to $\infty$ since $x_t\underset{t\rightarrow \infty}{\longrightarrow} 0$. 
For the first term, recall that $f$ is constant over $[0,\epsilon]$ for some $\epsilon>0$.
One has $f'(x_t)=0$ for large enough $t$ hence the drift part in the immigration and the branching components vanish. 
Let $\Psi^0$ be the branching mechanism $\Psi^0(u)=\Psi(u)-bu$. One has 
    \begin{align*}
        &|\mathcal{A}^{(t)}f(x_t)-\mathcal{A}f(0)|=\\
        &\left\lvert t\mathrm{L}^{\Psi^0}(f\circ g)(g^{-1}(x_t))+ \int_{x_t}^{\infty}f'(v)t\bar{\nu}(g^{-1}(v)-g^{-1}(x_t))\ddr v-\int_{0}^{\infty}f'(v)\frac{\ddr v}{v} \right \lvert \\
        &\leq \left\lvert t\mathrm{L}^{\Psi^0}(f\circ g)(g^{-1}(x_t))\right \lvert+         \left\lvert \int_{\epsilon}^{\infty}f'(v)t\bar{\nu}(g^{-1}(v)-g^{-1}(x_t))\ddr v-\int_{\epsilon}^{\infty}f'(v)\frac{\ddr v}{v} \right \lvert\\
        &\leq \left\lvert t\mathrm{L}^{\Psi^0}(f\circ g)(g^{-1}(x_t))\right \lvert +\left\lvert \int_{\epsilon}^{\infty}f'(v)\big(t\bar{\nu}(g^{-1}(v)-g^{-1}(x_t))-1/v\big)\ddr v\right \lvert.   \end{align*}
        By Lemma \ref{lem: fastjumpdiff}, the second summand above goes towards $0$ and we are left to study $t\mathrm{L}^{\Psi^0}(f\circ g)(g^{-1}(x_t))$. 
 One has $(f\circ g)'(g^{-1}(x_t))=f'(x_t)g'(g^{-1}(x_t))=0$ for large enough $t$, thus
\begin{align}
&t\mathrm{L}^{\Psi^0}(f\circ g)(g^{-1}(x_t))\nonumber\\ &=tg^{-1}(x_t)\int_0^\infty \Big(f\circ g(g^{-1}(x_t)+u)-f\circ g (g^{-1}(x_t))-u(f\circ g)'(g^{-1}(x_t))\Big)\pi(\ddr u)\nonumber \\
&=tg^{-1}(x_t)\int_0^\infty \Big(f\circ g(g^{-1}(x_t)+u)-f\circ g (g^{-1}(x_t))\Big)\pi(\ddr u)\nonumber \\
&=tg^{-1}(x_t)\int_0^C \Big(f\circ g(g^{-1}(x_t)+u)-f\circ g (g^{-1}(x_t))\Big)\pi(\ddr u) \label{uleq1}\\
\qquad \qquad &\qquad \qquad+tg^{-1}(x_t)\int_C^\infty \Big(f\circ g(g^{-1}(x_t)+u)-f\circ g (g^{-1}(x_t))\Big)\pi(\ddr u) \label{ugeq1}. 
\end{align}
Note that $f\circ g(g^{-1}(x_t))=f(x_t)$ which is constant for $t$ large enough since $x_t$ goes to $0$.  For \eqref{uleq1}, observe first that since $g^{-1}$ is increasing $g^{-1}(x_t)\leq g^{-1}(\epsilon/2)$ for $t$ large enough. The function $g^{-1}$ being rapidly varying in the variable $t$, see \eqref{defg}, one has $g^{-1}(\epsilon/2)/g^{-1}(\epsilon)\underset{t\rightarrow \infty}{\longrightarrow} 0$. Moreover since $g^{-1}(\epsilon)$ goes to $\infty$ as $t$ tends to $\infty$ and $u$ is bounded by $C$,  $g^{-1}(x_t)+u\leq g^{-1}(\epsilon)$ for $t$ large enough. Therefore \eqref{uleq1} is identically zero for $t$ large enough. For \eqref{ugeq1} we conclude that it converges towards $0$ as in Step 4, see the end of Section \ref{sec:convawayfrom0},
using that $\sup_{y\geq 0,v\geq 0}|tyg'(y+v)|< \sup_{y\geq 0}|tyg'(y)|<\infty$ by \eqref{eqn:firstconv}. 

\subsection{Proof of Theorem \ref{thmdivergencecase}}\label{sec:proofmaintheorem}

By combining the convergences obtained in \eqref{supepsilontoinfinity} and  \eqref{supepsilonto0} above on $[\epsilon,\infty)$ and $[0,\epsilon]$, we get
\begin{equation}\label{convAt0} \|\mathcal{A}^{(t)}f(x)-\mathcal{A}f(x)\|\underset{t\rightarrow \infty}{\longrightarrow} 0
\end{equation}
for any $f\in C^{2}_c([0,\infty))$ such that $f$ is constant near $0$. 

We now show that for any $f\in \mathcal{D}_{\mathrm{ESN}}$, there is a family of functions $(f_t)$ in $C^{2}_c([0,\infty))$, constant near $0$ and such that 
\[f_t\underset{t\rightarrow \infty}{\longrightarrow} f \text{ and } \mathcal{A}^{(t)}f_t \underset{t\rightarrow \infty}{\longrightarrow} \mathcal{A}f \text{ uniformly}.\]
We appeal to Lemma \ref{lemmaapprox}. 
The latter provides a sequence $(f_n)$ converging uniformly towards $f$ and
satisfying $\epsilon_n:=||\mathcal{A}f_n-\mathcal{A}f||\underset{n\rightarrow\infty}{\longrightarrow} 0$. Thus by \eqref{convAt0} for any large $n$, there exists $t(n)>0$ such that for $t\geq t(n)$, $\| \mathcal{A}^{(t)}f_n-\mathcal{A}f_n\|\leq \epsilon_n$, hence 
\begin{align*}
\|\mathcal{A}^{(t)}f_n-\mathcal{A}f\|\leq \| \mathcal{A}^{(t)}f_n-\mathcal{A}f_n\|+\epsilon_n \leq 2\epsilon_n.  
\end{align*}
Choosing $(t(n),n\geq 1)$ strictly increasing and going to $\infty$, we finally set $f_t=f_n$ for any $t\in [t(n),t(n+1))$ and claim that
\[f_t\underset{t\rightarrow \infty}{\longrightarrow} f \text{ and } \mathcal{A}^{(t)}f_t \underset{t\rightarrow \infty}{\longrightarrow} \mathcal{A}f \text{ uniformly}.\]
We can now apply \cite[Theorem 19.25-(i-iv)]{Kallenberg} and state 
\begin{equation}\label{convergenceprocessvarphi} 
\left(\frac{1}{t\varphi(1/Y_{st})},s\geq 0\right)\underset{t\rightarrow \infty}{\Longrightarrow} (M_s,s\geq 0),
\end{equation}
with $M$ an ESN$(b/c,\mu)$ and $\mu(\ddr x):=\ddr x/x^2$. It only remains to verify that we can replace $\varphi$ by $\Phi$. This follows from the following lemma which also shows that one can exchange $\Phi$ or $\varphi$ by any equivalent function at $0$ that is increasing, see the last statement of Theorem \ref{thmdivergencecase}.
\begin{lemma}
[Convergence up to equivalence] \label{lem:equivconvergence} Let $F$ be increasing and equivalent to $\varphi$ at $0$. Then
\begin{equation}\label{unifFphi}
\underset{0\leq s\leq T}{\sup} \left| \frac{1}{tF(1/Y_{st})} - \frac{1}{t\varphi(1/Y_{st})}\right| \underset{t\rightarrow \infty}{\Longrightarrow} 0.
\end{equation}
\end{lemma}
\begin{proof}
We first show that for any $\eta\in(0,1)$, we have 
for large enough $t$,
\begin{equation}\label{eqn:phidiff} \left| \frac{1}{tF(1/y)} - \frac{1}{t\varphi(1/y)}\right|\leq \frac{\eta}{t\varphi(1/y)}\vee 2\eta,\quad \forall y\geq 0.\end{equation}
Let $\eta>0$ be fixed. Let $y\geq 0$. If $\frac{1}{t\varphi(1/y)}\leq \eta,$ then $1/y\geq \varphi^{-1}(1/t\eta)$, and since $F$ is increasing, $F(1/y)\geq F(\varphi^{-1}(1/t\eta))$. On the other hand, since $F$ and $\varphi$ are equivalent at $0$, for $t$ large enough,  $F(\varphi^{-1}(1/t\eta))\geq \frac{1}{2} \varphi(\varphi^{-1}(1/t\eta))=1/2t\eta$ and then
$$\frac{1}{tF(1/y)}\leq 2\eta.$$ 
We deduce plainly that in the case $\frac{1}{t\varphi(1/y)}\leq \eta$, the left hand side in \eqref{eqn:phidiff} is smaller than $2\eta$.

If now $\frac{1}{t\varphi(1/y)}\geq \eta$ (i.e. $\varphi(1/y)\leq \frac{1}{t\eta}$), then by taking $t$ large enough, we have $1/y$ small enough. Using the equivalence of $\varphi$ and $F$ at $0$,
large enough $t$ leads to
$$\frac{\varphi(1/y)}{F(1/y)}\in[1-\eta,1+\eta].$$
Factorising the left hand side of \eqref{eqn:phidiff} by $\frac{1}{t\varphi(1/y)}$, we easily deduce that it is smaller than $\frac{\eta}{t\varphi(1/y)}$.
Then we conclude that \eqref{eqn:phidiff} is proved. 
Applying it to $y=Y_{st}$ provides
that for any $T>0$, $\eta\in (0,1)$,  for $t$ large enough,  
\begin{equation}\label{upperbounduniform}
\underset{0\leq s\leq T}{\sup} \left| \frac{1}{tF(1/Y_{st})} - \frac{1}{t\varphi(1/Y_{st})}\right| \leq \underset{0\leq s\leq T}{\sup} \frac{\eta}{t\varphi(1/Y_{st})}\vee 2\eta. \end{equation}
On the other hand, the weak convergence \eqref{convergenceprocessvarphi} on the time interval $[0,T]$ entails the convergence in law \[\underset{0\leq s\leq T}{\sup} \frac{1}{t\varphi(1/Y_{st})} \underset{t\rightarrow \infty}{\longrightarrow} \underset{0\leq s\leq T}{\sup} M_s<\infty,\] 
see e.g. Whitt and Sleeman \cite[Theorem 6.1]{zbMATH03667676}. This ensures that the upper bound in \eqref{upperbounduniform} is arbitrarily small and we have established \eqref{unifFphi}.
\qed 
\end{proof}
\noindent \underline{Denouement}: 
By choosing $F=\Phi$ which meets all requirements of Lemma \ref{lem:equivconvergence} and combining the weak convergence \eqref{convergenceprocessvarphi} with $\varphi$ and the uniform control \eqref{eqn:phidiff}, we get by applying \cite[Theorem 3.1]{zbMATH01354815},
that $\left(\frac{1}{t\Phi(1/Y_{st})},s\geq 0\right)$ converges weakly to the same limit process in the Skorokhod sense. The proof of Theorem \ref{thmdivergencecase} is finished. \qed

We explain now Remark \ref{remarknoconv}.
    In the so-called Sub-log case, namely when $x\Phi(e^{-x})\underset{x\rightarrow \infty}{\longrightarrow} 0$, the convergence of generators in Section \ref{sec:convawayfrom0} fails, indeed the limiting drift term \eqref{pointwisedrift} (with $c=0$) blows up. There is therefore no weak convergence in this setting as convergence of generators is necessary, see \cite[Theorem 19.25]{Kallenberg}. 

\subsection{Proof of Theorem \ref{thmextremasubordinator}}

We start by establishing the second statement (ii) namely the convergence of the renormalised process in Skorokhod weak sense. The latter is obtained along similar arguments as in the CBI case. 
\subsubsection{Statement (ii): convergence in Skorokhod sense}
Recall that we work under the assumption $\mathbb{H}$.  By applying Lemma \ref{eqn:equifun} there is a strictly increasing function $\varphi$ equivalent to $\Phi$ at $0$ satisfying the properties listed in Lemma \ref{lem:twoderivezero}  and in Lemma \ref{lem: fastjumpdiff}, that is to say: by setting $g(y):=\frac{1}{t\varphi(1/y)}$, we have for any $c_0>0$,   
\begin{equation}\label{lemcombinedsub}
\bar\nu(g^{-1}(v)-g^{-1}(x))\underset{t\rightarrow \infty}{\longrightarrow} \bar{\mu}(v)=1/v \text{ uniformly in }v\geq x+c_0 \text{ with } x\geq 0, \end{equation} 
and
\begin{equation}\label{lemdriftsub} tg'(y)=\frac{\varphi'(1/y)}{y^2\varphi(1/y)^2}\underset{y\rightarrow \infty}{\longrightarrow} 0.
\end{equation}
We now study the convergence of the generators.  Recall $\mathrm L^{\Phi}$ the generator of the subordinator with Laplace exponent $\Phi$, see \eqref{immigrationpartgen} and $\mathcal{A}$ that of the extremal process $Z$ with $\bar{\mu}(x)=1/x$. Let $\mathcal{A}^{(t)}$ be the generator of the process $\left(g(Y_{st}),s\geq 0\right)$. Then for any $f\in C^{1,0}([0,\infty))\cap \mathcal{D}_{\mathrm{ESN}}$, $f\circ g$ belongs to the domain of the subordinator and $f$ to that of the extremal process and 
\begin{align*}\mathcal{A}^{(t)}f(x)&=t\mathrm L^{\Phi}(f\circ g)(g^{-1}(x))\\
&=t\int_0^\infty \left[(f\circ g)(g^{-1}(x)+u)-f(x)\right]\nu(\ddr u)+\beta t(f\circ g)'(g^{-1}(x))\\
&=t\int_x^\infty f'(v)\bar\nu(g^{-1}(v)-g^{-1}(x))\ddr v+\beta tf'(x)g'(g^{-1}(x)).
\end{align*}
For any $\epsilon>0$, the first term on the right side in the last equality converges to $\mathcal{A}f(x):=\int_x^\infty f'(v)\frac{\ddr v}{v}$ uniformly in $x\geq \epsilon$ using \eqref{lemcombinedsub} and the same method as in Step 2 in Section \ref{sec:convawayfrom0}. For the second term, we have by \eqref{lemdriftsub}
$$tg'(g^{-1}(x))=tg'(y)=\frac{\varphi'(1/y)}{y^2\varphi(1/y)^2}\underset{y\to\infty}{\longrightarrow}0.$$
Let $x\geq\epsilon$ and set $y:=g^{-1}(x)$,  since $y\geq g^{-1}(\epsilon)\underset{t\to\infty}{\longrightarrow}\infty$. The convergence above holds uniformly in $x\geq \epsilon>0$. For $0\leq x\leq \epsilon$, we can proceed similarly as in Section \ref{secconv0} for the CBI case and establish the uniform convergence for a function constant near $0$. Finally any function in $\mathcal{D}_{\mathrm{ESN}}$ satisfies $f'(0)=0$ and can thus be approximated by a $C^1$ function constant near zero, see e.g. \cite[Lemma 2]{FoucartYuanESN23}. We conclude that $\mathcal{A}^{(t)}f(x)$ converges uniformly in $x\geq 0$ to $\mathcal{A}f(x)$ as $t$ goes to $\infty$ for any $f\in \mathcal{D}_{\mathrm{ESN}}$ (which is a core for the extremal process, see Lemma \ref{lem:coreESN}).  As previously, \cite[Theorem 19.25-(i-iv)]{Kallenberg} ensures the weak convergence 
\begin{equation*}
\left(\frac{1}{t\varphi(1/Y_{st})},s\geq 0\right)\underset{t\rightarrow \infty}{\Longrightarrow} (M_s,s\geq 0).
\end{equation*}
Finally, the fact that we can replace $\varphi$ by any increasing function $F$ equivalent to $\varphi$ at $0$ (in particular, one can take $F=\Phi$) follows from the same argument as in Lemma \ref{lem:equivconvergence}.
\qed
\subsection{Statement (i): convergence in finite dimensional sense.}
Recall, see e.g. \cite{Res87}, that if $M$ is an extremal process based on a PPP whose intensity has for tail $\bar{\mu}(x)=1/x$ then, for any $0\leq s_1\leq ...\leq s_n$ and $(x_1,...,x_n) \in \mathbb{R}^n$,
\begin{equation}\label{extremalprocess}
\mathbb{P}(M_{s_1}\leq y_1,M_{s_2}\leq y_2,...,M_{s_n}\leq y_n)=F^{s_1}(y'_1)F^{s_2-s_1}(z'_2)...F^{s_n-s_{n-1}}(y'_n)
\end{equation}
where $y'_i=\wedge_{k=i}^{n} y_{k}$ for all $i\geq 1$ and $F$ is the probability distribution function $F(y)=e^{-1/y}$ for $y\in (0,\infty)$.\\
We start by the two-dimensional convergence. The multi-dimensional convergence is explained in Step 2 and follows essentially from the same argument. 

\noindent \underline{Step 1}: two-dimensional convergence. Let $\theta_1,\theta_2>0$, $s_2>s_1>0$ and $x_1,x_2>0$.
\begin{align*}
\mathbb{E}[e^{-\theta_1\Phi^{-1}\big(x_1/t\big)Y_{s_1t}}e^{-\theta_2\Phi^{-1}\big(x_2/t\big)Y_{s_2t}}]&=\mathbb{E}[e^{-\big(\theta_1\Phi^{-1}\big(x_1/t\big)+ \theta_2\Phi^{-1}\big(\frac{x_2}{t}\big)\big)Y_{s_1t}}e^{-\theta_2\Phi^{-1}\big(\frac{x_2}{t}\big)(Y_{s_2t}-Y_{s_1t})}]\\
&=\mathbb{E}[e^{-\big(\theta_1\Phi^{-1}\big(x_1/t\big)+ \theta_2\Phi^{-1}\big(\frac{x_2}{t}\big)\big)Y_{s_1t}}]\mathbb{E}[e^{-\theta_2\Phi^{-1}\big(\frac{x_2}{t}\big)(Y_{(s_2-s-1)t})}]\\
&=e^{-s_1t\Phi\big(\theta_1\Phi^{-1}(x_1/t)+\theta_2\Phi^{-1}(x_2/t)\big)}e^{-(s_2-s_1)t\Phi\big(\theta_2\Phi^{-1}(x_2/t)\big)}.
\end{align*}
We now study the limit as $t$ goes to $\infty$. Since $\Phi$ is slowly varying at $0$, and for any $\lambda>0$, $\Phi^{-1}(\lambda/t)\longrightarrow 0$ as $t\to \infty$, we have 
\begin{equation}\label{eqn:phithetaphi}\frac{t}{x_2}\Phi(\theta \Phi^{-1}(x_2/t))=\frac{\Phi(\theta \Phi^{-1}(x_2/t))}{\Phi( \Phi^{-1}(x_2/t))}\underset{t\rightarrow \infty}{\longrightarrow} 1.\end{equation}
Thus $(s_2-s_1)t\Phi\big(\theta_2\Phi^{-1}(x_2/t)\big)\underset{t\rightarrow \infty}{\longrightarrow} (s_2-s_1)x_2$. Moreover $\Phi$ is an increasing slowly varying function, therefore the inverse function $\Phi^{-1}$ is increasing and rapidly varying at $0$, namely, 
\begin{equation}\label{eqn:phiorder}\text{if } 
 x_2>x_1, \text{ then } \Phi^{-1}(x_2/t)/\Phi^{-1}(x_1/t)\underset{t\rightarrow \infty}{\longrightarrow} 0.\end{equation}
Hence, we see that $$t\Phi\big(\theta_1\Phi^{-1}(x_1/t)+\theta_2\Phi^{-1}(x_2/t)\big)\underset{t\rightarrow \infty}{\sim} t\Phi\big(\Phi^{-1}(x_1\vee x_2/t)\big)= x_1\vee x_2.$$
Therefore
\begin{align*}
\underset{t\rightarrow \infty}{\lim} \mathbb{E}[e^{-\theta_1\Phi^{-1}\big(x_1/t\big)Y_{s_1t}}e^{-\theta_2\Phi^{-1}\big(x_2/t\big)Y_{s_2t}}]&=\exp\big(-s_1x_1\vee x_2-(s_{2}-s_{1})x_2\big).
\end{align*}
 The convergence of the Laplace transform with $\theta_1,\theta_2 >0$ arbitrary provides that  $$\left(\Phi^{-1}\big(x_1/t\big)Y_{s_1t},\Phi^{-1}\big(x_2/t\big)Y_{s_2t}\right)\underset{t\rightarrow \infty}{\longrightarrow} (W^{x_1}(s_1),W^{x_2}(s_2)) \text{ in law}.$$
Moreover since the limit does not depend on $\theta_1$ and $\theta_2$, the limiting distribution is degenerated (i.e. is supported by $\{0,\infty\}^{2}$) and $$\mathbb{P}(W^{x_1}(s_1)=0,W^{x_2}(s_2)=0)=\exp\big(-s_1x_1\vee x_2-(s_{2}-s_{1})x_2\big).$$ 
To conclude on the convergence of the two-dimensional laws, observe that
\begin{align*}
\mathbb{P}(t\Phi(1/Y_{s_1t})\geq x_1,t\Phi(1/Y_{s_2t})\geq x_2)&=\mathbb{P}(Y_{s_1t}\leq 1/\Phi^{-1}(x_1/t),Y_{s_2t}\leq 1/\Phi^{-1}(x_2/t))\\
&=\mathbb{P}(\Phi^{-1}(x_1/t)Y_{s_1t}\leq 1,\Phi^{-1}(x_2/t)Y_{s_2t}\leq 1)\\
 &\quad \underset{t\rightarrow \infty}{\longrightarrow}\exp\big(-s_1x_1\vee x_2-(s_{2}-s_{1})x_2\big).
\end{align*}
Finally, by setting $y_i=1/x_i$ for $i=1,2$ we get
\begin{align*}
\mathbb{P}\left(\frac{1}{t\Phi(1/Y_{s_1t})}\leq y_1,\frac{1}{t\Phi(1/Y_{s_2t})}\leq y_2\right)\underset{t\rightarrow \infty}{\longrightarrow} \mathbb{P}(M(s_1)\leq y_1,M(s_2)\leq y_2)=e^{-s_1\frac{1}{y_1}\wedge \frac{1}{y_2}}e^{-(s_2-s_1)\frac{1}{y_2}}
\end{align*}
which matches with \eqref{extremalprocess} for $n=2$.\\

\noindent \underline{Step 2}: multi-dimensional convergence. Let $n\in\N.$ Let $0=s_0<s_1<s_2<\cdots<s_n$, $x_1>0,\ldots,x_n>0$ and $\theta_1,\ldots, \theta_n>0$. Then we have 

\begin{align}
\mathbb{E}\left[\prod_{j=1}^n e^{-\theta_j\Phi^{-1}\big(x_j/t\big)Y_{s_jt}}\right]&=\E\left[\exp\left(-\sum_{i=1}^n\left(\sum_{j=i}^n \theta_j\Phi^{-1}\big(x_j/t\big)\right)(Y_{s_it}-Y_{s_{i-1}t})\right)\right] \nonumber\\
=\exp&\left( -\sum_{i=1}^n(s_i-s_{i-1})t\Phi\left(\sum_{j=i}^n \theta_j\Phi^{-1}\big(x_j/t\big)\right)\right)\nonumber\\
&\underset{t\rightarrow \infty}{\longrightarrow} \exp\left( -(s_n-s_{n-1})x_n-\sum_{i=1}^{n-1}(s_i-s_{i-1})(x_i\vee x_{i+1}\vee\cdots\vee x_n) \right)\nonumber\\
 &\qquad= \mathbb{P}(W^{x_1}(s_1)=0,W^{x_2}(s_2)=0,\cdots, W^{x_n}(s_n)=0), \label{multidimW}
\end{align}
where we used repeatedly \eqref{eqn:phithetaphi} and \eqref{eqn:phiorder} to get the limit and $W^{x_j}(s_j)$ stands for a random variable valued in $\{0,\infty\}$ limit in law of $\Phi^{-1}\big(x_j/t\big)Y_{s_jt}$ as $t$ goes to $\infty$. Now let $y_j=1/x_j$ and write
\[\bigcap_{j=1}^{n}\{1/t\Phi(1/Y_{s_jt})\leq y_j\}=\bigcap_{j=1}^{n}\{1\geq \Phi^{-1}(x_j/t)Y_{s_jt}\}.\]
We see that
\begin{align*}
\mathbb{P}(M_{s_1}\leq y_1,\ldots, M_{s_n}\leq y_n)&=\underset{t\rightarrow \infty}{\lim} \mathbb{P}\left(\bigcap_{j=1}^{n}\{1\geq \Phi^{-1}(x_j/t)Y_{s_jt}\}\right)=\mathbb{P}(W^{x_1}(s_1)=0,\cdots, W^{x_n}(s_n)=0),
\end{align*}
and the identity \eqref{multidimW} matches with \eqref{extremalprocess}. This allows us to conclude. \qed

\textbf{Acknowledgements} The authors thank Chunhua Ma for discussions. C.F and L.Y are supported  by the French National Research Agency (ANR): LABEX MME-DII (ANR11-LBX-0023-01). C.F is supported by the European Union (ERC, SINGER, 101054787). Views and opinions expressed are however those of the authors only and do not necessarily reflect those of the European Union or the European Research Council. Neither the European Union nor the granting authority can be held responsible for them.
\bibliographystyle{amsalpha}

%
%
%
%
%
%
%
%
%
%
%
%
%
%
%
%

\end{document}